\documentclass[10pt]{article}
\usepackage[utf8]{inputenc}
\usepackage[english]{babel}
\usepackage{amssymb,amsmath,amsthm}
\usepackage{indentfirst}
\usepackage{enumitem}
\usepackage{tocloft}
\usepackage{MnSymbol}
\usepackage{mathtools}

\theoremstyle{definition}

\newtheorem{theorem}{Theorem}[section]

\newtheorem{proposition}[theorem]{Proposition}
\newtheorem{lemma}[theorem]{Lemma}

\newtheorem{remark}[theorem]{Remark}

\newtheorem{notation}[theorem]{Notation}

\title{$\mathbb{Z}_2$-graded  polynomial identities for the Jordan algebra of $2\times 2$ upper triangular matrices}
\author{
Dimas Jos\'e Gon\c{c}alves
\\
 Universidade Federal de S\~ao Carlos (UFSCar) \\
Departamento de Matem\'{a}tica \\
13565-905 S\~ao Carlos, SP, Brasil\\
e-mail: \texttt{dimasjg@ufscar.br}\\
\\
Mateus Eduardo Salom\~ao \\
 Universidade Federal de S\~ao Carlos (UFSCar) \\
Departamento de Matem\'{a}tica \\
13565-905 S\~ao Carlos, SP, Brasil\\
e-mail: \texttt{m.salomaoutfpr@gmail.com}
}

\begin{document}

\maketitle



\

\begin{abstract}
Let $K$ be a field (finite or infinite) of char$(K)\neq 2$ and let $UT_n=UT_n(K)$ be the  $n\times n$ upper triangular matrix algebra over $K$. 
If $\cdot $ is the usual product on $UT_n$ then with the new product $a\circ b=(1/2)(a\cdot b +b\cdot a)$ we have that $UT_n$
is a Jordan algebra, denoted by $UJ_n=UJ_n(K)$. In this paper, we describe the set of all $\mathbb{Z}_2$-graded 
polynomial identities of $UJ_2$ with any nontrivial $\mathbb{Z}_2$-grading. Moreover, we describe a linear basis for the 
corresponding relatively free $\mathbb{Z}_2$-graded algebra.
\end{abstract}

\section{Introduction}

Let $K$ be a field and let $UT_n=UT_n(K)$ be the algebra of $n\times n$ upper triangular matrices over 
$K$.  This algebra plays an important role in PI-Theory, and its polynomial identities were described in 
\cite{Maltsev,Siderov}.

With respect to gradings, let $G$ be any group. The $G$-gradings on $UT_n$ were described as follows: In \cite{Valenti}, it was proved that every $G$-grading on $UT_n$ is isomorphic to an elementary $G$-grading
; in \cite{vicova}, 
the elementary $G$-gradings were classified. In \cite{vicova}, it was proved that two $G$-gradings on $UT_n$ are isomorphic if and only if
they satisfy the same $G$-graded polynomial identities. Moreover, the set of all $G$-graded polynomial identities of $UT_n$ was described as follows: in \cite{vicova} when $K$ is an infinite field, in \cite{evandro} when $K$ is a finite field.


From now on, we assume char$(K)\neq 2$.
Denote by $UJ_n=UJ_n(K)$ the vector space $UT_n$ with a new product $\circ $ given by
\[u\circ v=(1/2)(u\cdot v +v\cdot u)\]
where $u,v \in UT_n$. Then $UJ_n$ is a Jordan algebra. 

In \cite{faco}, the polynomial identities of $UJ_2$ were described when $K$ is an infinite field of char$(K)\neq 2,3$.
It is an open problem to describe the polynomial identities of $UJ_n$ when $n\geq 3$. 

With respect to gradings, let $G$ be any group. 
All $\mathbb{Z}_2$-gradings on $UJ_2$ were described in \cite{faco}. After,
in \cite{feco} it was proved that  if $K$ is infinite, then 
every $G$-grading on $UJ_n$ is, up to a graded isomorphism,  either elementary or MT (mirror type). 
Moreover, in \cite{feco} the authors proved that 
two $G$-gradings on $UJ_n$ are isomorphic if and only if
they satisfy the same $G$-graded polynomial identities. 

In \cite{faco}, Koshlukov and Martino described the set of all  
$\mathbb{Z}_2$-graded 
polynomial identities of $UJ_2$ when $K$ is any field of characteristic $0$ and, as a consequence, 
in \cite{centromartinosouza} it was proved that the variety of Jordan algebras generated by $UJ_2$ endowed with 
any $G$-grading has the Specht property.

Since almost nothing is known concerning polynomial identities of Jordan algebras, and 
motivated by the results above, given any nontrivial $\mathbb{Z}_2$-grading on $UJ_2$ we describe the set of its 
$\mathbb{Z}_2$-graded 
polynomial identities when $K$ is any field (finite or infinite).
Moreover, we describe a linear basis for the 
corresponding relatively free $\mathbb{Z}_2$-graded algebra when $K$ is any field (finite or infinite). 
In order to obtain our descriptions we use some ideas from the paper \cite{faco}.


We draw the reader's attention to the fact that gradings and graded polynomial identities for $UT_n$ were studied in another context too: Lie algebra. See, for example, \cite{ feco2}.

\section{Preliminaries} 

Throughout this section,  $K$ is a field (finite or infinite) with $\mbox{char}(K) \neq 2$. 

Let $UT_n=UT_n(K)$ be the  $n\times n$ upper triangular matrix algebra over $K$. On this algebra we consider the 
usual product $\cdot$ and so it is an associative algebra. Denote by $UJ_n=UJ_n(K)$ the vector space $UT_n$ with a new product $\circ $ given by
\[u\circ v=(1/2)(u\cdot v +v\cdot u)\]
where $u,v \in UT_n$. Then $UJ_n$ is a Jordan algebra. 

Let $e_{ij}$ be the matrix unit in $UJ_n$ whose $(i,j)$th entry equals $1$ and all other entries equal $0$. 
We denote
\[1=e_{11}+e_{22}, \ a=e_{11}-e_{22} \ \mbox{and} \ b=e_{12}.\]
Note that 
\[a\circ a=1 \ \mbox{and} \ a\circ b=b\circ b =0.\]
For convenience, if $u,v\in UJ_n$ we will write $u\circ v = uv$. 

Let $\mathbb{Z}_2=\{0,1\}$ and denote by 
\[UJ_2=(UJ_2)_0\oplus (UJ_2)_1\]
a $\mathbb{Z}_2$-grading on $UJ_2$. Thus $(UJ_2)_i\circ (UJ_2)_j \subseteq (UJ_2)_{i+j}$ for all $i,j \in \mathbb{Z}_2$.

The next proposition was proved in \cite{faco}.

\begin{proposition}\label{graddeJ}
The following decompositions $UJ_2 = (UJ_2)_0 \oplus (UJ_2)_1$ are $\mathbb{Z}_2$-gradings on $UJ_2(K):$
\begin{enumerate}
\item The associative grading: $(UJ_2)_0 = K \oplus Kb$, $(UJ_2)_1 = Ka$;
\item The scalar grading: $(UJ_2)_0 = K$, $(UJ_2)_1 = Ka \oplus Kb$;
\item The classical grading: $(UJ_2)_0 = K \oplus Ka$, $(UJ_2)_1 = Kb$;
\item The trivial grading: $(UJ_2)_0=UJ_2$, $(UJ_2)_1=0$;
\end{enumerate}
where we identify $K$ with the scalar matrices in $UJ_2$. The four gradings are pairwise nonisomorphic. 
They are, up to a $\mathbb{Z}_2$-graded isomorphism, the only $\mathbb{Z}_2$-gradings on $UJ_2$.  
\end{proposition}
\begin{proof}
See \cite[Lemma 2, Lemma 3, Proposition 4]{faco}.
\end{proof}

Let $Y=\{y_1,y_2,\ldots \}$ and $Z=\{z_1,z_2,\ldots \}$ be disjoint infinite sets. If $X=Y \cup Z$, denote by $J(X)$
the  free $\mathbb{Z}_2$-graded Jordan algebra, freely generated by $X$ over $K$. We remember that $J(X)$ is the 
free Jordan algebra, freely generated by $X$ over $K$, and 
\[J(X)=(J(X))_0\oplus (J(X))_1\]
is its $\mathbb{Z}_2$-grading where
$||y_i||=0$, $||z_i||=1$ for all $i$, and if $u,v$ are monomials in $J(X)$ then
\[||uv||=||u||+||v||.\]
Note that the notation $|| \ ||$  means the homogeneous degree, that is, if $f\in (J(X))_0$ then $||f||=0$ and $f$ is called even;  if $f\in (J(X))_1$ then $||f||=1$ and $f$ is called odd.

We remember that a $T_{\mathbb{Z}_2}$-ideal of $J(X)$ is an ideal of $J(X)$ closed under all $\mathbb{Z}_2$-graded endomorphisms
of $J(X)$.
If $W\subseteq J(X)$ we denote by $\langle W \rangle^{T_{\mathbb{Z}_2}}$ the $T_{\mathbb{Z}_2}$-ideal of $J(X)$ generated by $W$, that is, the minimal $T_{\mathbb{Z}_2}$-ideal of $J(X)$ containing $W$. If 
$f\in \langle W \rangle^{T_{\mathbb{Z}_2}}$ we say that $f$ is consequence of the polynomials in $W$.

By using similar arguments as in \cite[Proposition 4.2.3]{Drensky} we state the following:

\begin{proposition}\label{propositconseq}
Let $K$ be a field with $|K|$ elements. 
Let $f \in J(X)$, $w\in X$ and 
\[f=\sum_{i=0}^{d_w} f^{(i)} 
\]
where $f^{(i)}$ is the homogeneous component of $f$ with $\deg_{w} f^{(i)}=i$.
If $d_w < |K|$ then
\[\langle f \rangle^{T_{\mathbb{Z}_2}}=\langle f^{(0)},f^{(1)},\ldots,f^{(d_w)} \rangle^{T_{\mathbb{Z}_2}}.\]
\end{proposition}

Let $K[x_1,\ldots ,x_n]$ be the free commutative algebra, freely generated by $x_1,\ldots ,x_n$ over $K$. 
The next lemma is consequence of \cite[Proposition 4.2.3]{Drensky}.
\begin{lemma} \label{lemapolcomutativo}
Let $K$ be a field  with $|K|\geq q$. Given $f \in K[x_1,\ldots ,x_n]$ write
\[f(x_1, \ldots, x_n) = \sum_{d_1 = 0}^{q-1} \ldots \sum_{d_n = 0}^{q-1} \lambda_{(d_1, \ldots, d_n)}x_1^{d_1}\cdots x_n^{d_n},\]
where $\lambda_{(d_1, \ldots, d_n)} \in K$. If $f(\alpha_1, \ldots, \alpha_n)=0$ for all  
$\alpha_1, \ldots, \alpha_n \in K$ then $\lambda_{(d_1, \ldots, d_n)}=0$ for all $(d_1, \ldots, d_n)$.
\end{lemma}

\

If $UJ_2=(UJ_2)_0\oplus (UJ_2)_1$ is a $\mathbb{Z}_2$-grading and $f(y_1,\ldots ,y_s,z_1,\ldots ,z_n) \in J(X)$, we remember
that $f$ is a $\mathbb{Z}_2$-graded polynomial identity for $UJ_2$ if 
\[f(Y_1,\ldots ,Y_s,Z_1,\ldots ,Z_n)=0\]
for all $Y_1,\ldots ,Y_s \in (UJ_2)_0$ and $Z_1,\ldots ,Z_n \in (UJ_2)_1$. The set of all 
$\mathbb{Z}_2$-graded polynomial identities of $UJ_2$, denoted by $Id$, is a $T_{\mathbb{Z}_2}$-ideal of $J(X)$.
If the grading is the trivial, then $Id$ was described in \cite[Theorem 19]{faco} when $K$ is infinite with char$(K)\neq 2,3$.
We want to describe $Id$ when the grading is nontrivial. In this case, by Proposition \ref{graddeJ}, it is sufficient to consider
the associative, scalar and classical gradings only.  

If $u,v,w \in J(X)$ we denote by $(u,v,w)$ the associator, that is,
\[(u,v,w) = (uv)w - u(vw).\]
Note that
\begin{equation}\label{igualdadeassociator}
(w,v,u)=-(u,v,w) \ \mbox{and} \ (v,u,w)=(u,v,w)-(u,w,v).
\end{equation}

If $f_1, f_2, \ldots , f_n \in J(X)$ we use the following convention:
\[f_1f_2 \cdots f_{n-1} f_n=(f_1f_2 \cdots f_{n-1})f_n.\]

If $u,v,c,d \in J(X)$ then 
\begin{equation}\label{igualdadebasica}
uvcd + udcv + vdcu = (uv)(cd) + (uc)(vd) + (ud)(vc).
\end{equation}
In fact, the identity (\ref{igualdadebasica}) is true for any Jordan algebra (see \cite[Chapter I.7]{Jacobson}).
Moreover, 
renaming and comparing two expressions from (\ref{igualdadebasica}), we may also obtain
\begin{equation}\label{igualdadebasica2}
uvcd + udcv + vdcu = uvdc + ucdv + vcdu.
\end{equation}

\section{The associative grading }

Let $T_{\text{Ass}}(UJ_2)$ be the set of all $\mathbb{Z}_2$-graded polynomial identities for $UJ_2$ with the associative grading. 
In this section we will describe $T_{\text{Ass}}(UJ_2)$ for any field $K$ of char$(K)=p \neq 2$.

We remember that 
\[UJ_2=\left(UJ_2 \right)_0\oplus \left(UJ_2 \right)_1,\]
where 
\[\left(UJ_2 \right)_0 =span\{e_{11}+e_{22}, \ e_{12}\} \ \mbox{and} \ \left(UJ_2 \right)_1 =span\{e_{11}-e_{22}\}. \]

\begin{lemma}\label{lemaGA}
The polynomials 
\[(y_1,y_2,y_3) , (z_1, y_1, y_2),  (z_1, y_1 ,z_2), (z_1, z_2, z_3) \ \ \text{and} \ \ (z_1z_2, z_3, z_4)  \]
belong to $T_{\text{Ass}}(UJ_2)$.
\end{lemma}

\begin{proof}
The proof consists of a direct verification.
\end{proof}

\begin{notation}
Let $I$ be the $T_{\mathbb{Z}_2}$-ideal of $J(X)$ generated by the polynomials in Lemma \ref{lemaGA}. 
\end{notation}

Define the equivalence relation $\equiv$ on $J(X)$ as follows: if $f,g \in J(X)$ then
\[f\equiv g \Leftrightarrow f+I=g+I.\]

\begin{lemma}\label{lemadosdoisys}
The polynomials 
\[(z_1, y_1, y_2), \ (y_1, z_1, y_2) \ \mbox{and} \ (y_1, y_2, z_1) \]
belong to $I$.
\end{lemma}

\begin{proof} By definition we have $(z_1, y_1, y_2) \in I$. By (\ref{igualdadeassociator}),
\[(y_1, z_1, y_2)=(z_1,y_1,y_2)-(z_1,y_2,y_1)\in I \ \mbox{and} \  (y_1, y_2, z_1)=-(z_1,y_2,y_1)\in I\]
as desired.
\end{proof}

\begin{lemma}\label{lemadosdoiszs}
The polynomials 
\[(z_1z_2, x_1, x_2), \ (x_1, z_1z_2, x_2)  \ \ \text{and} \ \  (x_1,x_2, z_1z_2)\]
belong to $I$, where $x_1$ and $x_2$ are any variables in $Y\cup Z$.
\end{lemma}

\begin{proof}
By Lemma \ref{lemaGA} and Lemma \ref{lemadosdoisys} we obtain 
\[(z_1z_2, x_1, x_2)\in \langle (y_1,y_2,y_3) , ( y_1,z_1, y_2), ( y_1,y_2, z_1), (z_1z_2, z_3, z_4)\rangle^{T_{\mathbb{Z}_2}}\subseteq I.\]
By $(z_1z_2, x_1, x_2)\in I$ and (\ref{igualdadeassociator}) we obtain  $(x_1, z_1z_2, x_2), \  (x_1,x_2, z_1z_2)\in I$.
\end{proof}

\begin{lemma}\label{novaident}
The polynomial
\[((y_1y_2)z_1)z_2 - ((y_1z_1)z_2)y_2 - ((y_2z_1)z_2)y_1 + (z_1z_2)(y_1y_2)\]
belongs to $I$.
\end{lemma}

\begin{proof}
Let $f = ((y_1y_2)z_1)z_2 - ((y_1z_1)z_2)y_2 - ((y_2z_1)z_2)y_1 + (z_1z_2)(y_1y_2)$. 
If $u=y_1, \ v=z_1, \ c=z_2$ and $d=y_2$ in (\ref{igualdadebasica}) we obtain
\begin{align*}
f & = ((y_1y_2)z_1)z_2 + ((y_1y_2)z_2)z_1 - (z_1y_1)(z_2y_2) - (z_2y_1)(z_1y_2) \\
  & = ((y_1y_2)z_1)z_2 + ((y_1y_2)z_2)z_1 -  ((z_1y_1)y_2)z_2 - ((z_2y_1)y_2)z_1 + (z_1y_1, y_2, z_2) + (z_2y_1, y_2, z_1) \\
  & = (y_2, y_1, z_1)z_2 + (y_2, y_1, z_2)z_1 + (z_1y_1, y_2, z_2) + (z_2y_1, y_2, z_1).
\end{align*}
By Lemma \ref{lemadosdoisys} and $(z_1, y_1, z_2) \in I$ we conclude that $f \in I$. 
\end{proof}


\begin{lemma}\label{lemavariaveiscomutameassociam}
The subalgebras of $J(X)/I$ generated by the sets 
\[Y+I=\{y+I \ : \ y\in Y \} \ \mbox{ and } \ Z+I=\{z+I \ : \ z\in Z \}\]
 are commutative and associative. 
\end{lemma}

\begin{proof}
Let $A_Y$ and $A_Z$ be the subalgebras of $J(X)/I$ generated by the sets $Y+I$ and $Z+I$, respectively. The algebra
$J(X)$ is commutative, thus $A_Y$ and $A_Z$ are commutative too. 

Since  $(y_1,y_2,y_3) \in I$ we have that $A_Y$ is associative. 

Finally, let $f_1,f_2,f_3$ polynomials in the variables $z_1,z_2,  \ldots $ .
Note that $f_i$ is a sum of even and odd elements. We shall prove that $f=(f_1,f_2,f_3) \in I$. In this case, it is sufficient to suppose that each $f_i$ is either even or odd.
The next table shows that $f$ is consequence of $g \in I$.
\begin{center}
\begin{tabular}{|c|c|c|c|}
\hline 
$f_1$ & $f_2$ & $f_3$ &  $g$ \\ 
\hline 
Even & Even & Even & $(y_1,y_2,y_3)$ \\ 
\hline 
Even & Even & Odd & $(y_1,y_2,z_1)$ \\ 
\hline 
Even & Odd & Even & $(y_1,z_1 ,y_2)$ \\ 
\hline 
Even & Odd & Odd & $(z_1z_2, x_1, x_2)$, $(y_1,y_2,z_1)$ and $(y_1,y_2,y_3)$ \\ 
\hline 
Odd & Even & Even & $(z_1, y_1, y_2)$ \\ 
\hline 
Odd & Even & Odd & $(z_1, y_1 ,z_2)$ \\ 
\hline 
Odd & Odd & Even & $(z_1z_2, x_1, x_2)$, $(y_1,y_2,z_1)$ and $(y_1,y_2,y_3)$ \\ 
\hline 
Odd & Odd & Odd & $(z_1, z_2, z_3)$ \\ 
\hline 
\end{tabular} 
\end{center}

We will prove the fourth case. Let $f_1,f_2,f_3$ be monomials even, odd,  odd respectively. We will show by induction on deg$(f_1)$ that $f = (f_1,f_2,f_3)\in I$.
If deg$(f_1)=2$ then $f$ is consequence of $(z_1z_2, x_1, x_2)\in I$.  Suppose deg$(f_1) \geq 4$ and write $f_1 = f_1' f_1''$ where $f_1'$ and $f_1''$ are monomials with degree $<$ deg$(f_1)$. If $f_1'$ and $f_1''$ are odd, then $f$ is consequence of $(z_1z_2, x_1, x_2)\in I$. If $f_1'$ and $f_1''$ are even we obtain the equivalences below as follows:  by $(y_1,y_2,z_1) \in I$ we obtain ($\triangle 1$); by induction hypothesis we obtain ($\triangle 2$); by $(y_1,y_2,y_3) \in I$ we obtain ($\triangle 3$).
  \begin{align*}
((f_1'f_1'')f_2)f_3 & \overbrace{\equiv}^{(\triangle 1)}  (f_1'(f_1''f_2))f_3 \overbrace{\equiv}^{(\triangle 2)} f_1'((f_1''f_2)f_3) \overbrace{\equiv}^{( \triangle 2)} f_1'(f_1''(f_2f_3)) \overbrace{\equiv}^{(\triangle 3)} (f_1'f_1'')(f_2f_3).
\end{align*}
Thus $(f_1'f_1'',f_2,f_3) \equiv 0$ as desired.

Since $(f_1,f_2,f_3) = - (f_3,f_2,f_1)$ we obtain the seventh case. The other cases are trivial.
Therefore the subalgebra $A_Z$ is associative.
\end{proof}

\begin{lemma}\label{lemaZ}
Let $f, g \in J(X)$ and  $Z' = z_1 z_2 \cdots  z_s$, where $s$ is even. 
Then 
\[(fZ')g \equiv (fg)Z' .\]
\end{lemma}
\begin{proof}
Denote $Z''=z_1 z_2 \cdots  z_{s-1}$. By Lemma \ref{lemadosdoiszs} we have 
\begin{align*}
(fZ')g & \equiv (f(Z''z_s))g \equiv f((Z''z_s)g) \equiv f(g(Z''z_s)) \equiv (fg)(Z''z_s) \equiv (fg)Z'.
\end{align*}
The lemma is proved.
\end{proof}

Denote by $Sym(s)$ the symmetric group of $\{1,\ldots ,s\}$.

\begin{lemma}\label{lemaYz}
If $s$ is odd and $\sigma \in Sym(s)$, then
\[(y_1z_{\sigma(1)})(z_{\sigma(2)} \cdots z_{\sigma(s)}) \equiv (y_1z_1)(z_2 \cdots z_s).\]
\end{lemma}
\begin{proof}
By Lemma \ref{lemaZ} and Lemma \ref{lemavariaveiscomutameassociam} we have
\begin{align*}
(y_1z_{\sigma(1)})(z_{\sigma(2)} \cdots z_{\sigma(s)}) & \equiv y_1(z_{\sigma(1)}(z_{\sigma(2)} \cdots z_{\sigma(s)}))  \equiv y_1(z_1(z_2 \cdots z_s)) \\
 & \equiv  (y_1z_1)(z_2 \cdots z_s).
\end{align*}
The proof  is complete.
\end{proof}

\begin{lemma}\label{lemaYzz}
If $s$ is even and $\sigma \in Sym(s)$, then
\[(((y_1z_{\sigma(1)})z_{\sigma(2)})y_2)(z_{\sigma(3)} \cdots z_{\sigma(s)}) \equiv (((y_1z_1)z_2)y_2)(z_3 \cdots z_s).\]
\end{lemma}
\begin{proof}
Suppose $s=2$. Since $(z_1, y_1, z_2) \in I$, we obtain
\[((y_1z_1)z_2)y_2 \equiv ((z_1 y_1)z_2)y_2 \equiv (z_1(y_1z_2))y_2 \equiv
((y_1z_2)z_1)y_2.\]
If $s \geq 4$, 
denote $Z'=z_{\sigma(3)} \cdots z_{\sigma(s)}$. By Lemma \ref{lemaZ} we have 
\begin{align*}
 (((y_1z_{\sigma(1)})z_{\sigma(2)})y_2)Z' & \equiv  ((y_1z_{\sigma(1)})(z_{\sigma(2)}Z'))y_2                  
\end{align*}
and also
\begin{align*}
 (((y_1z_{\sigma(1)})z_{\sigma(2)})y_2)Z' & \equiv  ((y_1(z_{\sigma(1)}Z'))z_{\sigma(2)})y_2.
\end{align*}
Now we use the Lemma \ref{lemavariaveiscomutameassociam} to order the variables $z_1,\ldots , z_s$.
\end{proof}

\begin{lemma}\label{lemageradorGA}
Let $S$ be the subset of $J(X)$ formed by all polynomials
\begin{enumerate}
\item[(a)] $Y'Z'$,
\item[(b)] $(Y'z_{j_1})Z'$,
\item[(c)] $(((y_iz_{j_1})z_{j_2})Y')Z'$,
\end{enumerate}
where $Y'=y_{i_1} \cdots y_{i_r}$ with $r\geq 0$ and $i_1 \leq \ldots \leq i_r$;  $Z'=z_{l_1} \cdots z_{l_s}$ with
$s\geq 0$ even and  $j_1 \leq j_2 \leq l_1 \leq  \ldots \leq l_s$. 
Then the quotient vector space $J(X)/I$ is spanned by the set of all elements $g+I$ where $g\in S$.

\end{lemma}

\begin{proof}
 Let $A$, $B$ and $C$ be the sets of all elements $g+I$ where $g$ is in (a), (b) and (c), respectively. 
Denote $D=A\cup B \cup C$. 

\

\noindent {\bf Claim 1.} If $Y' = y_{k_1}y_{k_2} \cdots y_{k_m}$, $Y'' = y_{b_1}y_{b_2} \cdots y_{b_t}$ and 
$Z'=z_{l_1} \cdots z_{l_s}$ with
$s\geq 0$ even then $(((Y'z_{j_1})z_{j_2})Y'')Z'+I \in span D$.

\noindent \emph{Proof of the Claim 1.} The proof is by induction on $m$. If $m=0$, by 
Lemmas \ref{lemaZ} and  \ref{lemavariaveiscomutameassociam} we obtain
$((z_{j_1}z_{j_2})Y'')Z'+I \in A \subset span D$. If $m=1$, by  Lemmas \ref{lemavariaveiscomutameassociam} and \ref{lemaYzz} we obtain $(((y_{k_1}z_{j_1})z_{j_2})Y'')Z'+I \in C \subset span D$.

Suppose $m\geq 2$. 
By $(y_1,y_2,y_3) \equiv 0$ and Lemma \ref{novaident} we have
\begin{align*}
(((Y'z_{j_1})z_{j_2})Y'')Z' & \equiv ((((y_{k_1}y_{k_2} \cdots y_{k_m})z_{j_1})z_{j_2})Y'')Z'\\
& \equiv  (((((y_{k_1} \cdots y_{k_{m-1}})z_{j_1})z_{j_2})y_{k_m})Y'')Z'\\                           
                            & \ \ + ((((y_{k_m}z_{j_1})z_{j_2})(y_{k_1} \cdots y_{k_{m-1}}))Y'')Z' \\
                             & \ \ - (((z_{j_1}z_{j_2})(y_{k_1}y_{k_2} \cdots y_{k_m}))Y'')Z'\\
                            & \equiv ((((y_{k_1} \cdots y_{k_{m-1}})z_{j_1})z_{j_2})(y_{k_m}Y''))Z' \\
                            & \ \ +(((y_{k_m}z_{j_1})z_{j_2})((y_{k_1} \cdots y_{k_{m-1}})Y''))Z' \\
                            & \ \ - ((y_{k_1}y_{k_2} \cdots y_{k_m})Y'')((z_{j_1}z_{j_2})Z').
\end{align*}
By Lemmas \ref{lemavariaveiscomutameassociam} and \ref{lemaYzz} it follows that $(((y_{k_m}z_{j_1})z_{j_2})((y_{k_1} \cdots y_{k_{m-1}})Y''))Z' + I \in C \subset span D$ and $((y_{k_1}y_{k_2} \cdots y_{k_m})Y'')((z_{j_1}z_{j_2})Z')+I \in A \subset span D$.
By induction, $((((y_{k_1} \cdots y_{k_{m-1}})z_{j_1})z_{j_2})(y_{k_m}Y''))Z'+I \in span D$. The Claim 1 is proved.

Now, if $f$ is a monomial in $J(X)$, we shall prove by induction on deg $(f)$ that
$f+I \in span D$.

The cases deg $(f)=1$ and deg $(f)=2$ are trivial.

Suppose deg$(f)\geq 3$ and write $f=gh$ where 
$g,h \in J(X)$ are monomials with degree $<$ deg $(f)$.  By induction hypothesis it follows that 
$g+I$ and $h+I$ belong to $D$. We have six cases to consider: 

\begin{enumerate}
\item $g+I$ and $h+I$ belong to $A$. 

In this case, $g+I=Y'Z'+I$ and $h+I=Y''Z''+I$.  By Lemma \ref{lemaZ} we have  
\begin{align*}
f\equiv (Y'Z')(Y''Z'') & \equiv  ((Y'Z')Y'')Z'' \equiv ((Y'Y'')Z')Z'' \equiv (Y'Y'')(Z'Z'').
\end{align*}
By Lemma \ref{lemavariaveiscomutameassociam} it follows that $f+I \in A \subset span D$. 

\item $g+I$ belongs to $A$ and $h+I$ belongs to $B$.

In this case, $g+I=Y'Z'+I$ and $h+I=(Y''z_{j_1})Z''+I$. By Lemmas \ref{lemaZ} 
and \ref{lemadosdoisys} we have
\begin{align*}
f & \equiv (Y'Z')((Y''z_{j_1})Z'')  \equiv (Y'(Y''z_{j_1}))(Z'Z'') \equiv ((Y'Y'')z_{j_1})(Z'Z'').
\end{align*}
By Lemmas \ref{lemavariaveiscomutameassociam} and  \ref{lemaYz} it follows that $f+I \in B \subset span D$.

\item $g+I$ belongs to $A$ and $h+I$ belongs to $C$.

In this case, $g+I=Y'Z'+I$ and $h+I = (((y_iz_{j_1})z_{j_2})Y'')Z''+I$. By Lemma \ref{lemaZ} and $(y_1,y_2,y_3) \in I$ we have
\begin{align*}
f & \equiv (Y'Z')((((y_iz_{j_1})z_{j_2})Y'')Z'')  
   \equiv (Y'(((y_iz_{j_1})z_{j_2})Y''))(Z'Z'') \\
  & \equiv (((y_iz_{j_1})z_{j_2})(Y'Y''))(Z'Z'').
\end{align*} 
By Lemmas \ref{lemavariaveiscomutameassociam} and  \ref{lemaYzz} it follows that $f+I \in C \subset span D$.

\item $g+I$ and $h+I$ belong to $B$.

In this case, $g+I=(Y'z_{j_1})Z'+I$ and $h+I=(Y''z_{j_2})Z''+I$. 
By Lemma \ref{lemaZ}, by $(z_1,y_1,z_2) \in I$ and Lemma \ref{lemadosdoisys} we have
\begin{align*}
f & \equiv ((Y'z_{j_1})Z')((Y''z_{j_2})Z'') \equiv  ((Y'z_{j_1})(Y''z_{j_2}))(Z'Z'') \\
  & \equiv (((Y'z_{j_1})Y'')z_{j_2})(Z'Z'') \equiv  (((Y'Y'')z_{j_1})z_{j_2})(Z'Z'').
\end{align*}
By Claim 1 it follows that $f+I \in  span D$.

\item $g+I$ belongs to $B$ and $h+I$ belongs to $C$.

In this case, $g+I=(Y'z_{j_1})Z'+I$ and $h+I = (((y_iz_{j_2})z_{j_3})Y'')Z''+I$. We have the following congruences:
\begin{align*}
f & \equiv ((Y'z_{j_1})Z')((((y_iz_{j_2})z_{j_3})Y'')Z'') 
    \equiv [(Y'z_{j_1})(((y_iz_{j_2})z_{j_3})Y'')](Z'Z'')\\ 					  
  & \equiv [((Y'z_{j_1})Y'')((y_iz_{j_2})z_{j_3})](Z'Z'') 
    \equiv [((Y'Y'')z_{j_1})((y_iz_{j_2})z_{j_3})](Z'Z'') \\
  & \equiv [(Y'Y'')(z_{j_1}((y_iz_{j_2})z_{j_3}))](Z'Z'') 
    \equiv [(Y'Y'')((z_{j_1}z_{j_3})(y_iz_{j_2}))](Z'Z'') \\
  &  \equiv [(Y'Y'')(y_iz_{j_2})][(z_{j_1}z_{j_3})(Z'Z'')] 
   \equiv [((Y'Y'')y_i)z_{j_2}][(z_{j_1}z_{j_3})(Z'Z'')].
\end{align*}
By Lemmas \ref{lemavariaveiscomutameassociam} and \ref{lemaYz} it follows that  $f+I \in B \subset span D$.

\item $g+I$ and $h+I$ belong to $C$.

In this case, $g+I = (((y_iz_{j_1})z_{j_2})Y')Z' + I$ and $h+I = (((y_jz_{j_3})z_{j_4})Y'')Z''+ I$. 
We have the following congruences:
\begin{align*}
f & \equiv ((((y_iz_{j_1})z_{j_2})Y')Z')((((y_jz_{j_3})z_{j_4})Y'')Z'') \\
& \equiv  [(((y_iz_{j_1})z_{j_2})((y_jz_{j_3})z_{j_4}))(Y'Y'')](Z'Z'') \\
&  \equiv [(((y_iz_{j_1})(y_jz_{j_3}))(z_{j_2}z_{j_4}))(Y'Y'')](Z'Z'') \\
  &  \equiv [((y_iz_{j_1})(y_jz_{j_3}))(Y'Y'')][(z_{j_2}z_{j_4})(Z'Z'')] \\
  & \equiv 	[(((y_iy_j)z_{j_1})z_{j_3})(Y'Y'')][(z_{j_2}z_{j_4})(Z'Z'')]. 
\end{align*}
By Claim 1 it follows that  $f+I \in  span D$.
\end{enumerate}
The proof  is complete.
\end{proof}

\subsection{The associative grading, when $K$ is an infinite field }

In this subsection we describe the $\mathbb{Z}_2$-graded identities for $UJ_2$ with the associative grading when 
$K$ is infinite.

\begin{theorem}\label{teoremaprincipalGA}
If $K$ is an infinite field of char$(K)\neq 2$ then $I=T_{\text{Ass}}(UJ_2)$, that is, $T_{\text{Ass}}(UJ_2)$ is generated, as a $T_{\mathbb{Z}_2}$-ideal, by 
the polynomials in  Lemma \ref{lemaGA}. Moreover, the set in Lemma \ref{lemageradorGA} is a basis for the
quotient vector space $J(X)/I$.
\end{theorem}

\begin{proof}
By Lemma \ref{lemaGA} we have $I \subseteq T_{\text{Ass}}(UJ_2)$.

Let $S$ be the set in Lemma \ref{lemageradorGA} and write $\overline{S}=\{g+T_{\text{Ass}}(UJ_2): \ g\in S\}$.
Since $I \subseteq T_{\text{Ass}}(UJ_2)$ we have by Lemma \ref{lemageradorGA} that $J(X)/T_{\text{Ass}}(UJ_2)=span \overline{S}$.

We shall prove that $\overline{S}$ is a linearly independent set. 
Let
\[f=\sum_{g\in S} \lambda_g g \in T_{\text{Ass}}(UJ_2), \ \lambda_g \in K.\]
Since $K$ is an infinite field, by Proposition \ref{propositconseq} every multihomogeneous component of $f$ belongs to
$T_{\text{Ass}}(UJ_2)$. Thus it is sufficient to suppose the three cases below:

\begin{align*}
f & 
= \lambda (y_1^{k_1} \cdots y_r^{k_r})(z_1^{t_1} \cdots z_s^{t_s}) \\
 & \ \  +\sum_{i = 1}^r  \lambda_{i} (((y_iz_1)z_2)(y_1^{k_1} \cdots y_i^{k_i-1} \cdots y_r^{k_r}))(z_1^{t_1 - 1}z_2^{t_2-1}z_3^{t_3} \cdots z_s^{t_s})
\end{align*}
or
\begin{align*}
f & 
= \lambda (y_1^{k_1} \cdots y_r^{k_r})(z_1^{t_1} \cdots z_s^{t_s}) \\
 & \ \  +\sum_{i = 1}^r  \lambda_{i} (((y_iz_1)z_1)(y_1^{k_1} \cdots y_i^{k_i-1} \cdots y_r^{k_r}))(z_1^{t_1 - 2}z_2^{t_2}z_3^{t_3} \cdots z_s^{t_s})
\end{align*}
or 
\begin{align*}
f & 
= \lambda (y_1^{k_1} \cdots y_r^{k_r}z_1) (z_1^{t_1} \cdots z_s^{t_s} )
\end{align*}
where $t_1 + \ldots + t_s$ is even. We shall prove that $\lambda =
\lambda_i=0$ for all $i$.
Denote  
\[1 = e_{11}+e_{22}, \ a = e_{11}-e_{22} \  \mbox{and} \ b = e_{12}.\]
In the first and second cases, let $Y_i = \alpha_i 1 + \beta_i b$ and $Z_i = \gamma_i a$, where $\alpha_i, \beta_i,\gamma_i \in K$. 
 We have
\[f(Y_1, \ldots, Y_r, Z_1 , \ldots , Z_s) = \left[\begin{array}{cc}
A & B \\ 
0 & A
\end{array} \right] = 0,\]
where
\begin{align*}
A & = \left(\lambda +\sum_{i = 1}^r \lambda_i \right) \alpha_1^{k_1} \cdots \alpha_r^{k_r}\gamma_1^{t_1} \cdots \gamma_s^{t_s}=0; \\
B & = \sum_{i=1}^r \left(\lambda k_i + \lambda_i (k_i-1) + \sum_{\substack{j=1 \\ j \neq i}}^r \lambda_j k_i \right) \alpha_1^{k_1} \cdots \alpha_i^{k_i - 1}  \cdots \alpha_r^{k_r}\beta_i \gamma_1^{t_1} \cdots \gamma_s^{t_s} =0.
\end{align*}
Since $\alpha_i, \beta_i,\gamma_i$ are any elements of $K$, and  $K$ is infinite, we have by Lemma \ref{lemapolcomutativo} that
\begin{equation}\label{1}
\lambda + \sum_{i = 1}^r \lambda_i = 0
\end{equation}
and also 
\[\lambda k_i + \lambda_i (k_i-1) + \sum_{\substack{j=1 \\ j \neq i}}^r \lambda_j k_i =0 \]
for all $i=1, \ldots, r$ that is 
\begin{equation}\label{2}
\lambda k_i + \lambda_1 k_i + \cdots + \lambda_i (k_i-1) + \cdots + \lambda_r k_i = 0
\end{equation}
for all $i = 1, \ldots, r.$
By the equalities (\ref{1}) and (\ref{2}) we obtain the system
\[ \left\{\begin{array}{llllll}
\lambda + \lambda_1 + \lambda_2 + \ldots + \lambda_r = 0 \\ 
\lambda k_1 + \lambda_1 (k_1-1) + \lambda_2 k_1 + \ldots + \lambda_r k_1 = 0\\
\lambda k_2 + \lambda_1 k_2 + \lambda_2 (k_2-1) + \ldots + \lambda_r k_2 = 0 \\
\vdots\\
\lambda k_r + \lambda_1 k_r + \lambda_2 k_r + \ldots + \lambda_r (k_r-1) = 0
\end{array}\right.
\]
with 
augmented matrix
\[\left[ \begin{array}{ccccccc}
1 & 1 & 1 & \cdots &  1 & 0 \\ 
k_1 & k_1-1 & k_1 & \cdots &  k_1 & 0 \\ 
k_2 & k_2 & k_2-1 & \cdots & k_2 & 0 \\ 
\vdots & \vdots & \vdots & \ddots & \vdots & \vdots \\ 
k_r & k_r & k_r & \cdots & k_r-1 & 0
\end{array} \right] 
\sim
\left[ \begin{array}{ccccccc}
1 & 1 & 1 & \cdots &  1 & 0 \\ 
0 & 1 & 0 & \cdots &  0 & 0 \\ 
0 & 0 & 1 & \cdots & 0 & 0 \\ 
\vdots & \vdots & \vdots & \ddots & \vdots & \vdots \\ 
0 & 0 & 0 & \cdots & 1 & 0
\end{array} \right].
\]
Thus, $\lambda = \lambda_i = 0$ for all $i = 1, \ldots, r$.

In the third case, that is,
\begin{align*}
f(y_1, \ldots, y_r, z_1 , \ldots , z_s) & 
= \lambda (y_1^{k_1} \cdots y_r^{k_r}z_1) (z_1^{t_1} \cdots z_s^{t_s} )
\end{align*}
where $t_1 + \ldots + t_s$ is even, let $Y_i = 1$ and $Z_i = a$ for all $i$. Then
\[f(Y_1, \ldots, Y_r, Z_1 , \ldots , Z_s) = \lambda a =0\]
and so $\lambda = 0$.

Therefore, the set $\overline{S}$ is a basis for the quotient vector space $J(X)/T_{\text{Ass}}(UJ_2)$.
Moreover, since $I \subseteq T_{\text{Ass}}(UJ_2)$, by Lemma \ref{lemageradorGA} we have 
$I=T_{\text{Ass}}(UJ_2) $.
\end{proof}


\subsection{The associative grading, when $K$ is a finite field }

In this subsection we describe the $\mathbb{Z}_2$-graded identities for $UJ_2$ with the associative grading when 
$K$ is finite. Throughout this subsection, $K$ is a finite field with $|K| = q$ elements and char$(K)\neq 2$.

Since $(K-\{0\}, \cdot )$ is a group with $q-1$ elements it follows that 
$x^{q-1}=1$
for all $x\in K-\{0\}$. Therefore, 
$x^q=x$ for all $x\in K$. 

A direct consequence of this fact is the following lemma. 

\begin{lemma}\label{lemaGAf}
The polynomials 
\[(y_1^q - y_1)(y_2^q - y_2), \  \  z_1^q -z_1 \ \ \text{and} \ \ (y_1^q - y_1)z_1\]
belong to $T_{\text{Ass}}(UJ_2)$.
\end{lemma}

\begin{notation}
Let $I'$ be the $T_{\mathbb{Z}_2}$-ideal of $J(X)$ generated by the polynomials in Lemmas \ref{lemaGA} and \ref{lemaGAf}.
\end{notation}

\begin{lemma}\label{lemayq}
The polynomial $(y_1^q, x_1, x_2)$ belongs to $I'$, where $x_1$ and $x_2$ are any variables in $Y\cup Z$.
\end{lemma}

\begin{proof}
Since $(y_1,y_2,y_3)$, $(y_1,z_1 ,y_2)$, $(y_1, y_2, z_1) \in I \subseteq I'$ (see definition of $I$  and Lemma \ref{lemadosdoisys}) it follows that $(y_1^q, x_1, x_2) \in I'$ when $x_1\in Y$ or $x_2 \in Y$.

We will prove that $(y_1^q, z_1, z_2) \in I'$. 

\

\noindent {\bf Claim 1.} $(y_1^n, z_1, z_2) + I' = n[((y_1z_1)z_2)y_1^{n-1}] - n[(z_1z_2)y_1^n] + I'$ for all $n\geq 1$.

\noindent \emph{Proof of the claim}. The case $n=1$ is trivial. Suppose $n \geq 2$. 
By Lemmas \ref{novaident} and  \ref{lemaZ}, 
\begin{align*}
(y_1^n, z_1, z_2) + I' & = (y_1^n z_1)z_2 - y_1^n(z_1z_2) + I'\\
				  & = ((y_1^{n-1}z_1)z_2)y_1 + ((y_1z_1)z_2)y_1^{n-1} - (z_1z_2)y_1^n - y_1^n(z_1z_2) + I'\\
				  & = (y_1^{n-1},z_1,z_2)y_1 + ((y_1z_1)z_2)y_1^{n-1} - (z_1z_2)y_1^n  + I'.
\end{align*} 
Now we apply the
induction hypothesis on the first summand and $(y_1,y_2,y_3)\in I'$ to conclude the proof of the claim.

In particular, if $n=q$ then
\[(y_1^q, z_1, z_2) + I' = q[((y_1z_1)z_2)y_1^{q-1}] - q[(z_1z_2)y_1^q] + I' = I'\]
and the proof is complete.
\end{proof}

\begin{lemma}\label{lemazzyq}
The following equality is valid:
\[((y_iz_1)z_2)y_j^q + I' = ((y_iz_1)z_2)y_j + ((y_jz_1)z_2)y_i - (z_1z_2)(y_iy_j) + I'.\]
\end{lemma}

\begin{proof} 
Let $g = ((y_iz_1)z_2)y_j^q$. By Lemma \ref{lemayq} we have
\[g+I' = (y_iz_1)(z_2y_j^q) + I'.\] 
Since $(y_j^q - y_j)z \in I'$ we obtain
\[y_j^qz + I' = y_jz + I' .\]
We use this equality, $(z_1,y_1,z_2) \in I'$, $(z_1,y_1,y_2) \in I'$ and 
Lemma \ref{novaident} 
to obtain
\begin{align*}
g+I' =& (y_iz_1)(y_jz_2) + I' = ((y_iz_1)y_j)z_2 + I' = ((y_iy_j)z_1)z_2 + I' \\
=&((y_iz_1)z_2)y_j + ((y_jz_1)z_2)y_i - (z_1z_2)(y_iy_j) + I'
\end{align*} 
as desired.
\end{proof}

Denote by $\Lambda_n$ the set of all elements $(s_1, \ldots , s_n) \in \mathbb{Z}^n$ such that:
\begin{enumerate}
\item[a)] $0 \leq s_1, \ldots, s_n < 2q$;
\item[b)] If $s_i \geq q$ for some $i$, then $s_j < q$ for all $j \neq i$.
\end{enumerate}

\begin{lemma}\label{lemageradorGAf}
The quotient vector space $J(X)/I'$ is spanned by the set of all polynomials $g+I'$ such that
\begin{enumerate}
\item[(a)] $g=\overline{Y}$ or
\item[(b)] $g=Y_1'Z_1'$ or
\item[(c)] $g=(Y_1'z_{j})Z_2'$ or
\item[(d)] $g=(((y_iz_{j})z_{l})Y_1')Z_3'$
\end{enumerate}
where
\begin{itemize}
\item $\overline{Y}=y_1^{k_1} \cdots y_r^{k_r}$ with $(k_1, \ldots, k_r) \in \Lambda_r$ and $r\geq 1$ ; 
\item $Y_1' = y_1^{k_1} \cdots y_r^{k_r}$ with $0 \leq k_1, \ldots, k_r < q$ and $r\geq 1$;
\item $Z_1'=z_1^{t_1} \cdots z_s^{t_s}$ with $0 \leq t_1, \ldots, t_s < q$, $s \geq 1$ and $t_1 + \ldots + t_s > 0$ even;
\item $Z_2' = z_{j}^{t_{j}} z_{j+1}^{t_{j+1}} \cdots z_s^{t_s}$ with $j \geq 1$, $s \geq 1$, $0 \leq t_{j}< q-1$, $0 \leq t_{j + 1}, \ldots, t_s < q$ and $t_{j} + \ldots + t_s \geq 0$ even;
\item  $Z_3' = z_{l}^{t_{l}}z_{l+1}^{t_{l+1}} \cdots z_s^{t_s}$ 
with $1 \leq j \leq l$, $s \geq 1$,  $0 \leq t_{l + 1}, \ldots, t_s < q$ and $t_{l}+ t_{l+1} +\ldots + t_s \geq 0$ even.
Moreover, if $j<l$ then $0 \leq t_l <q-1$. If $j=l$ then $0 \leq t_l <q-2$.
\end{itemize}

\end{lemma}

\begin{proof} 
Firstly, with respect to $\overline{Y}, Y_1', Z_1', Z_2'$ and $Z_3'$, these polynomials are well defined. See Lemma \ref{lemavariaveiscomutameassociam}.

Let $A$, $B$, $C$ and $D$ be the sets of all elements $g+I'$ where $g$ is in (a), (b), (c) and (d) respectively. Denote $E = A \cup B \cup C \cup D$. Let $f$ be a monomial in $J(X)$, we shall prove that $f+I' \in span E$. 

Since $I \subseteq I'$ we have by Lemma \ref{lemageradorGA} that the quotient vector space $J(X)/I'$ is spanned by the set of all polynomials:
\begin{enumerate}
\item[(a')] $Y'Z'+I'$,
\item[(b')] $(Y'z_{j})Z''+I'$,
\item[(c')] $(((y_iz_{j})z_{l})Y')Z''+I'$, 
\end{enumerate}
where $Y'=y_1^{k_1}y_{2}^{k_{2}} \cdots y_r^{k_r}$;  $0 \leq k_1, \ldots, k_r$; $r \geq 1$; 
$Z'=z_1^{l_1}z_{2}^{l_2} \cdots z_s^{l_s}$; $0 \leq l_1, l_2,  \ldots, l_s$;  $s\geq 1$; $l_1 + l_2 + \ldots + l_s$ even; 
 $Z''=z_{l}^{t_l}z_{l+1}^{t_{l+1}} \cdots z_s^{t_s}$; $0 \leq  t_l, t_{l+1}, \ldots, t_s$; $1 \leq j \leq l$; $s\geq 1$; $ t_l + t_{l+1} + \ldots + t_s$ even.

Let $A'$, $B'$ and $C'$ be the sets of all elements in (a'), (b') and (c'), respectively. 
We shall prove that $A' \cup B' \cup C' \subseteq span E$.

\

\noindent \emph{Case 1.} $f+I' \in A'$.

In this case, $f + I' = (y_1^{k_1} \cdots y_r^{k_r})(z_1^{l_1} \cdots z_s^{l_s}) + I'$. 
Since $(y_i^q - y_i)(y_j^q - y_j) \in I'$ we obtain
\[y_i^q y_j^q + I' = y_i^q y_j + y_i y_j^q - y_iy_j + I' \ \text{and} \ y_i^{2q} + I' = 2 y_i^{q+1} - y_i^2 + I'.\]
By using several times these two equalities and Lemma \ref{lemavariaveiscomutameassociam}, we can suppose $(k_1, \ldots, k_r) \in \Lambda_r$. 

If $l_1 = \ldots = l_s = 0$ then $f+I' \in A \subset span E$.

Suppose  $l_i \neq 0$ for some $i$. Since $(z_i^q - z_i) \in I'$ we obtain 
\begin{equation} \label{aaa7}
z_i^q+I'=z_i+I'.
\end{equation}
By using several times this equality, we can suppose $0 \leq l_1, \ldots, l_s < q$. We have two cases to consider: $0\leq k_1, \ldots, k_r < q$ or $k_m \geq q$ for some $m$. In the first case, $f+I' \in B \subset span E$. In the second case, write 
$k_m = q + u_m$, where $0 \leq u_m < q$. Let 
\[j_1=\mbox{min}\{ j \ | \ l_j \geq 1 \}.\]
If $l_{j_1} - 1 \geq 1$  denote $j_2=j_1$; otherwise denote 
\[j_2=\mbox{min}\{ j \ | \ l_j \geq 1; \ j\neq j_1 \}.\]
 By Lemmas \ref{lemavariaveiscomutameassociam} and \ref{lemaZ}, we have
\[f+I' = [(y_m^q(z_{j_1}z_{j_2}))\underbrace{(y_1^{k_1} \cdots y_m^{u_m} \cdots y_r^{k_r})}_{Y_1'}]\underbrace{(z_{j_1}^{l_{j_1}-1}z_{j_2}^{l_{j_2} - 1} \cdots z_s^{l_s})}_{Z_3'} + I'.\]
Since $(y_m^q - y_m)z \in I'$ we obtain
\begin{equation} \label{aaa6}
y_m^qz + I' = y_mz + I' .
\end{equation}
By using this equality and 
Lemma \ref{lemayq} we obtain
\[f+I'=[((y_m^qz_{j_1})z_{j_2})Y_1']Z_3'+I'=[((y_mz_{j_1})z_{j_2}))Y_1']Z_3'+I' \in D \subset span E.\]
This case is finished.

\

\noindent \emph{Case 2}. $f+I' \in B'$.

In this case, $f+I' = ((y_1^{k_1} \cdots y_r^{k_r})z_j)(z_j^{t_j} z_{j+1}^{t_{j+1}}\cdots z_s^{t_s}) + I'$. 
By Lemma \ref{lemadosdoisys} and (\ref{aaa6}), we can suppose $0\leq k_1, \ldots, k_r < q $. By (\ref{aaa7}),
we can suppose $0 \leq t_j, t_{j+1}, \ldots, t_s < q$. 
If $t_j < q - 1$ then $f+I'\in C \subset span E$. If $t_j = q - 1$,  by Lemma \ref{lemavariaveiscomutameassociam},
Lemma \ref{lemaZ} and (\ref{aaa7}) we obtain
\begin{align*}
f+I' =&(\underbrace{(y_1^{k_1} \cdots y_r^{k_r})}_{Y_1'}z_j)(z_j^{q-1} \underbrace{z_{j+1}^{t_{j+1}}\cdots z_s^{t_s}}_{Z_2'}) + I'= (Y_1'z_j)(z_j^{q-1} Z_2') + I'\\
				  & =  (Y_1'z_j^q) Z_2' + I'=(Y_1'z_j) Z_2' + I' \in C \subset span E.
\end{align*} 

\noindent \emph{Case 3}. $f+I' \in C'$.

In this case, $f+I' = (((y_iz_j)z_l)\underbrace{(y_1^{k_1} \cdots y_r^{k_r})}_{Y_1'})(z_l^{t_l}\underbrace{z_{l+1}^{t_{l+1}} \cdots z_s^{t_s}}_{Z_3'})+I'$. By Lemma \ref{lemayq}, Lemma \ref{lemazzyq} and Case 1, we can suppose $0 \leq k_1, \ldots, k_r < q$.
By (\ref{aaa7}), we can suppose  $0 \leq t_l, \ldots, t_s < q$ too.  We have two cases to consider:

\begin{enumerate}
\item[(3.a)] $j < l$.

In this case, $f+I' = (((y_iz_j)z_l)Y_1')(z_l^{t_l}Z_3')+I'$. If $t_l= q-1$, by Lemma  \ref{lemaZ} and (\ref{aaa7}) we have 
\[f+I'  = (((y_iz_j)z_l^q)Y_1')Z_3'+I' 
	 =(((y_iz_j)z_l)Y_1')Z_3'+I'.
\]
Thus, we can suppose $0 \leq t_l < q - 1$.
In this case, $f+I' \in D \subset span E$.

\item[(3.b)] $j = l$.

In this case, $f+I' = (((y_iz_l)z_l)Y_1')(z_l^{t_l}Z_3')+I'$. If $t_l = q-1$, we can use the same argument as in Case 3.a. 
If $t_l = q-2$ then there exists $k>l$ such that $t_k\geq 1$, because $q-2$ is odd. Write 
$f+I' = (((y_iz_l)z_l)Y_1')((z_l^{q-2}z_k)Z_3'')+I'$. By Lemma  \ref{lemaZ} and (\ref{aaa7}) we obtain
\begin{align*}
f+I' & =   (((y_iz_l)(z_l^{q-1}z_k))Y_1')Z_3''+I' \\
	 & = (((y_iz_l^q)z_k)Y_1')Z_3''+I'=(((y_iz_l)z_k)Y_1')Z_3''+I'\in D \subset span E.
\end{align*}
Thus, we can suppose $0 \leq t_l < q - 2$.
In this case, $f+I' \in D \subset span E$. 
\end{enumerate}

The lemma is proved.
\end{proof}

\begin{theorem}
If $K$ is a finite field with $|K|=q$ elements and  char$(K)\neq 2$ then $I'=T_{\text{Ass}}(UJ_2)$, that is, $T_{\text{Ass}}(UJ_2)$ is generated, as a $T_{\mathbb{Z}_2}$-ideal, by 
the polynomials in  Lemmas \ref{lemaGA} and \ref{lemaGAf}. Moreover, the set in Lemma \ref{lemageradorGAf} is a basis for the
quotient vector space $J(X)/I'$.
\end{theorem}

\begin{proof}
By Lemmas \ref{lemaGA} and \ref{lemaGAf} we have $I' \subseteq T_{\text{Ass}}(UJ_2)$.

Denote by $S'$ the set of all polynomials $g$ in Lemma \ref{lemageradorGAf} - item (a).
Denote by $S''$ the set of all polynomials $g$ in Lemma \ref{lemageradorGAf} - items (b), (c), (d).
Let $S = S' \cup S''$ and $\overline{S}=\{g+T_{\text{Ass}}(UJ_2): \ g\in S\}$.
Since $I' \subseteq T_{\text{Ass}}(UJ_2)$, by Lemma \ref{lemageradorGAf} it follows that $J(X)/T_{\text{Ass}}(UJ_2)=span \overline{S}$.

We shall prove that $\overline{S}$ is a linearly independent set.
Let
\[f(y_1,\ldots ,y_r, z_1, \ldots , z_s)=\sum_{g\in S} \lambda_g g \in T_{\text{Ass}}(UJ_2), \ \lambda_g \in K.\]
In particular,
\[h=f(y_1,\ldots ,y_r, 0, \ldots , 0)=\sum_{g\in S'} \lambda_g g = \sum_{k \in \Lambda_r} \lambda_{k} y_1^{k_1} \cdots y_r^{k_r} \in T_{\text{Ass}}(UJ_2), \ k=(k_1,\ldots ,k_r).\]

Let $*$ be the involution on the associative algebra $UT_2$ defined by:
\[
\left(\begin{array}{cc}
a_{11}&a_{12} \\
0& a_{22}
\end{array}
\right)^*=
\left(\begin{array}{cc}
a_{22}&a_{12} \\
0& a_{11}
\end{array}
\right).
\]
Note that the symmetric elements of $UT_2$ form a vector subspace  with basis $\{e_{11}+e_{22},e_{12}\}$. 
Moreover, if $u,v \in (UJ_2)_0$ then
\[u\circ v = u\cdot v\]
where $\cdot$ is the usual product of $UT_2$. Thus $h=h(y_1,\ldots , y_r)$ is a $*$ -polynomial identity for $UT_2$ if 
$y_1, y_2, \ldots $ are symmetric variables.  
By \cite[Lemma 5.8]{ronalddimas}, we obtain $\lambda_{k}=0$ for all $k\in \Lambda_r$.

In particular, 
\[f(y_1,\ldots ,y_r, z_1, \ldots , z_s)=\sum_{g\in S''} \lambda_g g .\]
Write 
\[f=\sum_t f_t, \ t = (t_1, \ldots, t_s),\]
where  $f_t$ is multihomogeneous with respect to variables $z_1,\ldots ,z_s$  and $\deg_{z_i}f_t=t_i$ for all $i$.
Since $|K| = q$ and $\deg_{z_i}f < q$ for all $i$, by Proposition \ref{propositconseq} we have $f_t \in  T_{\text{Ass}}(UJ_2)$ for all $t$ and 

\begin{align}\label{aaa1}
f_t & = \sum_k \lambda_k (y_1^{k_1} \cdots y_r^{k_r}) (z_1^{t_1} \cdots z_s^{t_s}) \\
  & \ \ + \sum_{i = 1}^r \sum_k \lambda_{(i, k)} (((y_iz_1)z_2)(y_1^{k_1} \cdots y_i^{k_i} \cdots y_r^{k_r}))(z_1^{t_1 - 1}z_2^{t_2-1}z_3^{t_3} \cdots z_s^{t_s}) \nonumber
\end{align}
or
\begin{align}\label{aaa2}
f_t & = \sum_k \lambda_k (y_1^{k_1} \cdots y_r^{k_r})( z_1^{t_1} \cdots z_s^{t_s}) \\
  & \ \ + \sum_{i = 1}^r \sum_k \lambda_{(i, k)} (((y_iz_1)z_1)(y_1^{k_1} \cdots y_i^{k_i} \cdots y_r^{k_r}))(z_1^{t_1 - 2}z_2^{t_2} \cdots z_s^{t_s}) \nonumber
\end{align}
or
\begin{align}\label{aaa3}
f_t & = 
\sum_k \lambda_k ((y_1^{k_1} \cdots y_r^{k_r})z_1) (z_1^{t_1-1} z_2^{t_2} \cdots z_s^{t_s})
\end{align}
where $k = (k_1, \ldots, k_r)$, $0 \leq k_j < q$ for all $j$. We shall prove that $\lambda_k =
\lambda_{(i,k)}=0$ for all $k,i$.

Suppose $f_t$ as in (\ref{aaa1}).  Let $Y_i = \alpha_i 1 + \beta_i e_{12}$ and $Z_i = e_{11}-e_{22}$,
where $\alpha_i, \beta_i \in K$. We have
\[f_t(Y_1, \ldots, Y_r, Z_1 , \ldots , Z_s) = \left[\begin{array}{cc}
A & B \\ 
0 & A
\end{array} \right] = 0,\]
where

\begin{align*}
A & = \sum_k \lambda_k \alpha_1^{k_1} \cdots \alpha_r^{k_r} 
    + \sum_{i=1}^r \sum_{\substack{k \\ k_i < q - 1}}\lambda_{(i, k)} \alpha_1^{k_1} \cdots \alpha_i^{k_i + 1} \cdots \alpha_r^{k_r} 
    + \sum_{i=1}^r \sum_{\substack{k \\ k_i = q - 1}} \lambda_{(i, k)} \alpha_1^{k_1} \cdots \alpha_i^q \cdots \alpha_r^{k_r}; \\
B & = \sum_k \lambda_k \left(\sum_{j=1}^r k_j \alpha_1^{k_1} \cdots \alpha_j^{k_j - 1}\beta_j \cdots\alpha_r^{k_r}  \right)\\
   & \ \  + \sum_{i=1}^r \sum_{\substack{k \\ k_i < q - 1}} \lambda_{(i, k)} \left( k_i \alpha_1^{k_1} \cdots \alpha_i^{k_i}\beta_i \cdots\alpha_r^{k_r} + \sum_{\substack{j=1 \\ j\neq i}}^r k_j \alpha_1^{k_1} \cdots \alpha_i^{k_i + 1} \cdots \alpha_j^{k_j - 1}\beta_j \cdots \alpha_r^{k_r} \right) \\
   & \ \ + \sum_{i=1}^r \sum_{\substack{k \\ k_i = q - 1}} \lambda_{(i, k)} \left((q - 1)\alpha_1^{k_1} \cdots \alpha_i^{q-1}\beta_i \cdots\alpha_r^{k_r} 
   + \sum_{\substack{j=1 \\ j\neq i}}^r k_j \alpha_1^{k_1} \cdots \alpha_i^q \cdots \alpha_j^{k_j - 1}\beta_j \cdots \alpha_r^{k_r}\right).
\end{align*}

Since $B=0$ for all $\alpha_1, \ldots , \alpha_r, \beta_1, \ldots , \beta_r \in K$ and $\deg_{\beta_i}B=1<q$ for all $i$, it follows that 
every homogeneous component of $B$ with respect to $\beta_i$ of degree $1$ is zero too. Thus, $B=B_1+ \ldots +B_r$ where
\begin{align*}
B_i=& \ \ \sum_k \lambda_k k_i \alpha_1^{k_1} \cdots \alpha_i^{k_i - 1}\beta_i \cdots\alpha_r^{k_r} 
	+ \sum_{\substack{k \\ k_i < q - 1}} \lambda_{(i, k)}k_i \alpha_1^{k_1} \cdots \alpha_i^{k_i}\beta_i \cdots \alpha_r^{k_r}\\
	& \ + \sum_{\substack{l=1 \\ l \neq i}}^r \sum_{\substack{k \\ k_l < q - 1}} \lambda_{(l, k)} k_i \alpha_1^{k_1} \cdots 
	\alpha_l^{k_l+1} \cdots  \alpha_i^{k_i-1}\beta_i \cdots  \alpha_r^{k_r} 
	+  \sum_{\substack{k \\ k_i = q - 1}} \lambda_{(i, k)} (q-1) \alpha_1^{k_1} \cdots \alpha_i^{q-1}\beta_i \cdots \alpha_r^{k_r} \\
	&  \ + \sum_{\substack{l=1 \\ l \neq i}}^r \sum_{\substack{k \\ k_l = q - 1}} \lambda_{(l, k)} k_i \alpha_1^{k_1} \cdots
	\alpha_l^q \cdots \alpha_i^{k_i-1}\beta_i \cdots  \alpha_r^{k_r} = 0.
\end{align*}
and $B_1=\ldots =B_r=0$.
Since $B_i=0$ for all $\alpha_1, \ldots , \alpha_r, \beta_i \in K$ and $\deg_{\alpha_i}B_i <q$ it follows that 
every homogeneous component of $B_i$ with respect to $\alpha_i$ is zero too. Thus, $B_i=B_{i,0}+B_{i,1}+ \ldots +B_{i,q-1}$ where $\deg_{\alpha_i} B_{i,j}=j$, and
\[B_{i,q-1}=\sum_{\substack{k \\ k_i = q - 1}} \lambda_{(i, k)} (q-1) \alpha_1^{k_1} \cdots \alpha_i^{q-1}\beta_i \cdots \alpha_r^{k_r} = 0.\]
Since $B_{i,q-1}$ is a polynomial identity for $K$ and $\deg_{\alpha_j} B_{i,q-1}<q$ for all $j$, by Lemma \ref{lemapolcomutativo} we obtain   $\lambda_{(i, k)} = 0$ for all $i$ and $k = (k_1, \ldots, k_{i-1}, q-1, k_{i+1}, \ldots, k_r)$. 

In particular, by (\ref{aaa1}) we have 
\begin{align*}
f_t & = \sum_k \lambda_k y_1^{k_1} \cdots y_r^{k_r} z_1^{t_1} \cdots z_s^{t_s} \\
  & \ \ + \sum_{i = 1}^r \sum_{\substack{k \\ k_i < q - 1}} \lambda_{(i, k)} (((y_iz_1)z_2)(y_1^{k_1} \cdots y_i^{k_i} \cdots y_r^{k_r}))(z_1^{t_1 - 1}z_2^{t_2-1}z_3^{t_3} \cdots z_s^{t_s}).
\end{align*}
Since $\deg_{y_j}f_t <q$ and $\deg_{z_j}f_t <q$ for all $j$, every multihomogeneous component of $f_t$ belongs to 
$T_{\text{Ass}}(UJ_2)$. 
Now, using similar arguments as in Theorem \ref{teoremaprincipalGA} it follows that $\lambda_k = \lambda_{(i,k)} = 0$
for all $i,k$ as desired.

The second case (\ref{aaa2}) is analogous to the first (\ref{aaa1}).

Now, let $f_t$ as in (\ref{aaa3}), that is, 
\[f_t  = \sum_k \lambda_k ((y_1^{k_1} \cdots y_r^{k_r})z_1) (z_1^{t_1-1} \cdots z_s^{t_s}).\]
Since $\deg_{y_j}f_t< q$ for all $j$, it follows that 
\[f_{t,k}(y_1, \ldots, y_r, z_1, \ldots, z_s) = \lambda_k ((y_1^{k_1} \cdots y_r^{k_r})z_1)( z_1^{t_1-1} \cdots z_s^{t_s} ) \in T_{\text{Ass}}(UJ_2)\]
for all $k$. Thus, if $Y_i = 1$ and $Z_i = e_{11}-e_{22}$ then
\[0=f_{t,k}(Y_1, \ldots, Y_r, Z_1, \ldots, Z_s) = \lambda_k (e_{11}-e_{22}) \]
and so $\lambda_k = 0$ for all $k$. 

Therefore, the set $\overline{S}$ is a basis for the quotient vector space $J(X)/T_{\text{Ass}}(UJ_2)$.
Moreover, since $I' \subseteq T_{\text{Ass}}(UJ_2)$, by Lemma \ref{lemageradorGAf} we have 
$I'=T_{\text{Ass}}(UJ_2) $.
\end{proof}

\section{The scalar grading }

Let $T_{\text{Sca}}(UJ_2)$ be the set of all $\mathbb{Z}_2$-graded polynomial identities for $UJ_2$ with the scalar grading. 
In this section we will describe $T_{\text{Sca}}(UJ_2)$ for any field $K$ of char$(K)=p \neq 2$.

We remember that 
\[UJ_2=\left(UJ_2 \right)_0\oplus \left(UJ_2 \right)_1,\]
where 
\[\left(UJ_2 \right)_0 =span\{e_{11}+e_{22}\} \ \mbox{and} \ \left(UJ_2 \right)_1 =span\{e_{11}-e_{22}, \ e_{12}\}. \]

\begin{lemma}\label{lemaGS}
The polynomials
\[(y_1,y_2,y_3) ,  \ (z_1, y_1, y_2), \ (y_1, z_1,z_2) \ \text{and} \ z_1(z_2,z_3,z_4)\]
belong to $T_{\text{Sca}}(UJ_2)$.
\end{lemma}
\begin{proof}
The proof consists of a direct verification.
\end{proof}

\begin{notation}
Let $I$ be the $T_{\mathbb{Z}_2}$-ideal of $J(X)$ generated by the polynomials in Lemma \ref{lemaGS}. 
\end{notation}

Define the equivalence relation $\equiv$ on $J(X)$ as follows: if $f,g \in J(X)$ then
\[f\equiv g \Leftrightarrow f+I=g+I.\]

\begin{lemma}\label{lemaconsequenciasGS}
Let $x_1, x_2, x_3 \in Y\cup Z$. 

\begin{itemize}
\item[(a)] If $x_i\in Y$ for some $1\leq i \leq 3$, then 
$(x_1,x_2,x_3)\in I$.
\item[(b)] $z_1(x_1,x_2,x_3)\in I$.
\end{itemize}
\end{lemma}

\begin{proof}
By $(y_1,y_2,y_3) , (z_1, y_1, y_2), (y_1, z_1,z_2) \in I$ and (\ref{igualdadeassociator}) we prove (a). 
 By $z_1(z_2,z_3,z_4) \in I$ and  (a) we prove (b).
\end{proof}

\begin{lemma}\label{lemapermutarz}
If $s$ is even and $\sigma \in Sym(s)$ then
\[(z_{\sigma(1)}z_{\sigma(2)})(z_{\sigma(3)}z_{\sigma(4)})\cdots (z_{\sigma(s-1)}z_{\sigma(s)})
\equiv (z_1z_2)(z_3z_4) \cdots (z_{s-1}z_s).\]
\end{lemma}

\begin{proof}
Since $(z_{\sigma (i)}z_{\sigma (i+1)})$ is an even polynomial, by Lemma \ref{lemaconsequenciasGS} - (a) it is sufficient to prove 
$(z_1z_2)(z_3z_4) \equiv (z_1z_3)(z_2z_4)$.
By Lemma \ref{lemaconsequenciasGS} we have  
\begin{align*}
(z_1z_2)(z_3z_4)&\equiv z_1(z_2(z_3z_4)) \equiv z_1((z_2z_3)z_4) \equiv z_1((z_3z_2)z_4)\\ 
&\equiv z_1(z_3(z_2z_4))  \equiv (z_1z_3)(z_2z_4).
\end{align*} 
The proof is complete. 
\end{proof}
 
\begin{lemma}\label{lemageradorGS}
Let $S$ be the subset of $J(X)$ formed by all polynomials
\begin{enumerate}
\item[(a)] $Y'Z'$ and
\item[(b)] $Y'(z_{i_0}Z')$,
\end{enumerate}
where $Y'=y_1^{k_1} \cdots y_r^{k_r}$; $k_i \geq 0$ for all $i$; $r\geq 0$;  $Z'=(z_{i_1}z_{j_1})(z_{i_2}z_{j_2}) \cdots (z_{i_t}z_{j_t})$;  
$i_1 \leq j_1 \leq i_2 \leq j_2 \leq \ldots \leq i_t \leq j_t$; $t\geq 0$; $i_0 \geq 1$. 
Then  the quotient vector space $J(X)/I$ is spanned by the set of all elements $g+I$ where $g\in S$. 
\end{lemma}

\begin{proof}
Let $A$ and $B$ be the sets of all elements $g+I$ where $g$ is in (a) and (b), respectively. Denote $C=A\cup B$. 
If $f(y_1, \ldots, y_r, z_1, \ldots, z_s)$ is a monomial in $J(X)$, we shall prove that $f+I \in span C$.

By Lemma \ref{lemaconsequenciasGS} - (a) we have
\begin{equation} \label{aaa4}
f \equiv y_1^{k_1} \cdots y_r^{k_r} g(z_1, \ldots, z_s),
\end{equation}
where $g(z_1, \ldots, z_s)$ is a monomial in the variables $z_1,z_2, \ldots ,z_s$.

Let $A_Z$ be the subalgebra of $J(X)/I$ generated by the set $\{z+I \ : \ z\in Z \}$.

\

\noindent {\bf Claim 1}. The vector space $A_Z$ is spanned by all elements $Z''+I$ and $z_{i_0}Z''+I$
where $Z''=(z_{i_1}z_{j_1})(z_{i_2}z_{j_2}) \cdots (z_{i_t}z_{j_t})$ and $t\geq 0$.

\noindent \emph{Proof of the Claim 1}. Let $h \in A_Z$ be a monomial. We will prove the result by induction on deg$(h)=n$. For $n = 1,2,3$ is trivial. Suppose $n \geq 4$ and write $h=h_1h_2$ where deg$(h_1),$ deg$(h_2)$ $<n$. We use the 
induction hypothesis on $h_1$ and $h_2$, and  
by Lemma \ref{lemaconsequenciasGS} - (a) we have the desired.

By (\ref{aaa4}), Claim 1 and Lemma \ref{lemapermutarz} we finish the proof of the lemma.
\end{proof}

\subsection{The scalar grading, when $K$ is an infinite field }

In this subsection we describe the $\mathbb{Z}_2$-graded identities for $UJ_2$ with the scalar grading when 
$K$ is infinite.

\begin{theorem}\label{teoremaprincipalGS}
If $K$ is an infinite field of char$(K)\neq 2$ then $I=T_{\text{Sca}}(UJ_2)$, that is, $T_{\text{Sca}}(UJ_2)$ is generated, as a $T_{\mathbb{Z}_2}$-ideal, by 
the polynomials in  Lemma \ref{lemaGS}. Moreover, the set in Lemma \ref{lemageradorGS} is a basis for the
quotient vector space $J(X)/I$.
\end{theorem}

\begin{proof}
If $g=g(z_1,\ldots, z_n)=z_{i_0}(z_{i_1}z_{j_1})(z_{i_2}z_{j_2}) \cdots (z_{i_t}z_{j_t})$ where $i_1 \leq j_1 \leq i_2 \leq j_2 \leq \ldots \leq i_t \leq j_t$, and 
\[d=(\deg_{z_1}g, \deg_{z_2}g, \ldots ,\deg_{z_n}g),\]
we denote $g=g_{(z_{i_0},d)}$.

 By Lemma \ref{lemaGS} we have $I \subseteq T_{\text{Sca}}(UJ_2)$.

Let $S$ be the subset in Lemma \ref{lemageradorGS} and  
write $\overline{S}=\{g+T_{\text{Sca}}(UJ_2): \ g\in S\}$.
Since $I \subseteq T_{\text{Sca}}(UJ_2)$ we have by Lemma \ref{lemageradorGS} that $J(X)/T_{\text{Sca}}(UJ_2)=span \overline{S}$.

We shall prove that $\overline{S}$ is a linearly independent set. 
Let
\[f(y_1,\ldots ,y_r,z_1,\ldots ,z_n)=\sum_{g\in S} \lambda_g g \in T_{\text{Sca}}(UJ_2), \ \lambda_g \in K.\]

Since $K$ is an infinite field, every multihomogeneous component of $f$ belongs to
$T_{\text{Sca}}(UJ_2)$. Thus it is sufficient to suppose

\[f = \lambda y_1^{k_1} \cdots y_r^{k_r}(z_{i_1}z_{j_1})(z_{i_2}z_{j_2}) \cdots (z_{i_t}z_{j_t})
\ \ \mbox{or} \  \ 
f= \sum_{i_0=1}^n \lambda_{i_0} y_1^{k_1} \cdots y_r^{k_r}  g_{(z_{i_0},d)},\]
where $i_1 \leq j_1 \leq i_2 \leq j_2 \leq \ldots \leq i_t \leq j_t$; $t\geq 0$. We have to prove that $\lambda =
\lambda_{i_0}=0$ for all $i_0$.
Denote  
\[1 = e_{11}+e_{22}, \ a = e_{11}-e_{22} \  \mbox{and} \ b = e_{12}.\]

In the first case, let $Y_i = 1$ and $Z_i = a$ for all $i$. Then
\[f(Y_1,\ldots,Y_r,Z_1, \ldots, Z_n) = \lambda 1=0\]
and so $\lambda = 0$.

In the second case, let $Y_i = 1$ for all $i$, $Z_{i_0} = a + b$, $Z_i=a$ for all $i\neq i_0$. 
Since $ab = b^2 = 0$ we obtain 
\[f(Y_1,\ldots,Y_r,Z_1, \ldots, Z_n) =(\lambda_1+\ldots +\lambda_n)a+\lambda_{i_0}b=0\]
and so $\lambda_{i_0}=0$ as desired.

Therefore, the set $\overline{S}$ is a basis for the quotient vector space $J(X)/T_{\text{Sca}}(UJ_2)$.
Moreover, since $I \subseteq T_{\text{Sca}}(UJ_2)$, by Lemma \ref{lemageradorGS} we have 
$I=T_{\text{Sca}}(UJ_2) $.
\end{proof}

\begin{remark}
There is a missing identity in the statement of \cite[Proposition 8]{faco}. But that missing identity was used in the proof of 
\cite[Proposition 8]{faco}. Here we give the complete list of these identities.
\end{remark}

\subsection{The scalar grading, when $K$ is a finite field }
In this subsection we describe the $\mathbb{Z}_2$-graded identities for $UJ_2$ with the scalar grading when $K$ is finite. Throughout this subsection, $K$ is a finite field with $|K| = q$ elements and char$(K)\neq 2$.

\begin{lemma}\label{lemapotenciadezgraduacaoescalar}
If $u = \alpha a + \beta b$, where $\alpha, \beta \in K$, $a = e_{11}-e_{22}$ and $b = e_{12}$, then
\begin{enumerate}
\item[i)] $u^{2n} = \alpha^{2n} 1,$
\item[ii)] $u^{2n-1} = \alpha^{2n-1}a+\alpha^{2n-2}\beta b,$
\end{enumerate}
for all $n \in \mathbb{N}$. In particular, $u^q = \alpha a + \alpha^{q-1}\beta b$.
\end{lemma}

\begin{proof}
By using induction on $n$ we can prove  i) and ii). Now, since $p\neq 2$ we have $p$ odd. Thus $q$ is odd and    
 \[u^q = \alpha^q a + \alpha^{q-1}\beta b=\alpha a + \alpha^{q-1}\beta b.\]
The lemma is proved.
\end{proof}

\begin{lemma}\label{lemaGSf}
The polynomials
\[\  y_1^q - y_1 \ \text{and} \ (z_1^q - z_1)z_2\]
belong to $T_{\text{Sca}}(UJ_2)$.
\end{lemma}
\begin{proof}
We will check the last polynomial only. Since $ab=b^2=0$, we have by Lemma \ref{lemapotenciadezgraduacaoescalar}
that $(z_1^q - z_1)z_2\in T_{\text{Sca}}(UJ_2)$.
\end{proof}

\begin{notation}
Let $I'$ be the $T_{\mathbb{Z}_2}$-ideal of $J(X)$ generated by the polynomials in Lemmas \ref{lemaGS} and \ref{lemaGSf}. 
\end{notation}

\begin{lemma}\label{lemageradorGSf}
Let $\widehat{S}$ be the subset of $J(X)$ formed by all polynomials
\begin{enumerate}
\item[(a)] $Y_1'Z_1'$ and
\item[(b)] $Y_1'(z_{i_0}Z_1')$,
\end{enumerate}
where $Y_1'=y_1^{k_1} \cdots y_r^{k_r}$; $0 \leq k_i < q$ for all $i$; $r\geq 0$;  $Z_1'=(z_{i_1}z_{j_1})(z_{i_2}z_{j_2}) \cdots (z_{i_t}z_{j_t})$; $i_1 \leq j_1 \leq i_2 \leq j_2 \leq \ldots \leq i_t \leq j_t $; $t\geq 0$; $0\leq \deg_{z_k}(Z_1')<q$ for all $k$; 
$i_0 \geq 1$. Then the quotient vector space $J(X)/I'$ is spanned by the set of all elements $g+I'$ where $g\in \widehat{S}$.

\end{lemma}

\begin{proof}
Let $A$ and $B$ be the sets of all elements $g+I'$ where $g$ is in (a) and (b), respectively. Denote $C=A\cup B$. 
If $f(y_1, \ldots, y_r, z_1, \ldots, z_s)$ is a monomial in $J(X)$, we shall prove that $f+I' \in span C$.

Since $I \subseteq I'$ we have by Lemma \ref{lemageradorGS} that the quotient vector space $J(X)/I'$ is spanned by the set of all polynomials:
\begin{enumerate}
\item[(a')] $Y'Z'+I'$,
\item[(b')] $Y'(z_{i_0}Z')+I'$,
\end{enumerate}
where $Y'=y_1^{k_1} \cdots y_r^{k_r}$; $k_i\geq 0 $ for all $i$; $r\geq 0$;  $Z'=(z_{i_1}z_{j_1})(z_{i_2}z_{j_2}) \cdots (z_{i_t}z_{j_t})$; $i_1 \leq j_1 \leq i_2 \leq j_2 \leq \ldots \leq i_t \leq j_t $; $t\geq 0$;  
$i_0 \geq 1$.

Let $A'$ and $B'$ be the sets of all elements in (a') and (b'), respectively. We shall prove that $A' \cup B' \subseteq span C$.

\

\noindent \emph{Case 1.} $f+I' \in A'$.

 \noindent In this case, $f+I' = (y_1^{k_1} \cdots y_r^{k_r})(z_{i_1}z_{j_1})(z_{i_2}z_{j_2}) \cdots (z_{i_t}z_{j_t}) + I'$.
Since $(y_i^q - y_i) \in I'$ we can suppose $0 \leq k_1, \ldots, k_r < q$.

Now,
\[\underbrace{(z_{k}z_{k}) \cdots (z_{k}z_{k})}_{(q-1)/2 \ \mbox{times}}(z_kz_l)=
z_{k}^{q-1}(z_kz_l)\]
and $z_{k}^{q-1}(z_kz_l)+I'=(z_{k}^{q-1}z_k)z_l+I'=(z_{k}^q)z_l+I'=z_kz_l+I'$ (see Lemma \ref{lemaGSf}).
Thus, we can suppose 
$0\leq \deg_{z_k}Z'<q$ for all $k$ and so $f+I'\in A \subseteq span C$.

\

\noindent \emph{Case 2.} $f+I' \in B'$.

\noindent We can use analogous argument to prove $f+I' \in B \subseteq span C$.

The lemma is proved. 
\end{proof}

\begin{theorem}\label{teoremaprincipalGS}
If $K$ is a finite field with $|K|=q$ elements and  char$(K)\neq 2$ then $I'=T_{\text{Sca}}(UJ_2)$, that is, $T_{\text{Sca}}(UJ_2)$ is generated, as a $T_{\mathbb{Z}_2}$-ideal, by 
the polynomials in  Lemmas \ref{lemaGS} and \ref{lemaGSf}. Moreover, the set in Lemma \ref{lemageradorGSf} is a basis for the
quotient vector space $J(X)/I'$.
\end{theorem}

\begin{proof}
By Lemmas \ref{lemaGS} and \ref{lemaGSf} we have $I' \subseteq T_{\text{Sca}}(UJ_2)$.

Consider the subset $\widehat{S}$ in  Lemma \ref{lemageradorGSf} and write 
$\overline{S}=\{g+T_{\text{Sca}}(UJ_2): \ g\in \widehat{S}\}$.
Since $I' \subseteq T_{\text{Sca}}(UJ_2)$ we have by Lemma \ref{lemageradorGSf} that $J(X)/T_{\text{Sca}}(UJ_2)=span \overline{S}$.
We shall prove that $\overline{S}$ is a linearly independent set. 

If $k=(k_1,\ldots ,k_r)$ we write $y_1^{k_1} \cdots y_r^{k_r}=y^{[k]}$. If $g=(z_{i_1}z_{j_1}) \cdots (z_{i_t}z_{j_t})$,
$\deg_{z_i}g=d_i$ and $d=(d_1,\ldots ,d_n)$ then we write $g=z^{[d]}$.
Let
\[f(y_1,\ldots , y_r,z_1, \ldots ,z_n)=\sum_{g\in \widehat{S}} \lambda_gg\in T_{\text{Sca}}(UJ_2),\]
where $\lambda_g \in K$.
We will prove that $\lambda_g=0$ for all $g$. Since $\deg_{y_i}f<q$ for all $i$ we can suppose
\[f=\lambda y^{[k]}+\sum_{d}\lambda_d y^{[k]}z^{[d]}+\sum_{i=1}^n\sum_{d}\lambda_{(i,d)} y^{[k]}(z_iz^{[d]}).\]
By replacing $y_i$ and $z_i$ by the matrices $1$ and $0$ respectively, for all $i$, we obtain 
$0=f(1,\ldots,1,0,\ldots,0)=\lambda \cdot 1$ and so  $\lambda =0$. Now, since 
$T_{\text{Sca}}(UJ_2)$ is a $\mathbb{Z}_2$-graded ideal of $J(X)$ it follows that 
$f_1,f_2 \in  T_{\text{Sca}}(UJ_2)$ where 
\[f_1=\sum_{d}\lambda_d y^{[k]}z^{[d]} \ \ \mbox{and} \ \ f_2=\sum_{i=1}^n\sum_{d}\lambda_{(i,d)} y^{[k]}(z_iz^{[d]}).\]
Since $\deg_{z_i}f_1<q$ for all $i$ it follows that every $g_d=\lambda_d y^{[k]}z^{[d]} \in T_{\text{Sca}}(UJ_2)$. 
By replacing $y_i$ and $z_i$ by the matrices $1$ and $a=e_{11}-e_{22}$ respectively, for all $i$, we obtain 
$0=g_d(1,\ldots,1,a,\ldots,a)=\lambda_d \cdot 1$ and so  $\lambda_d =0$.

Denote $a = e_{11}-e_{22}$, $b = e_{12}$ and $Z_i = \alpha_i a + \beta_i b$, where $\alpha_i, \beta_i \in K$. We have
\[f_2(1, \ldots, 1, Z_1, \ldots, Z_n) = \alpha a + \beta b = 0, \ \mbox {where} \ \ \beta=\sum_{i=1}^n\sum_{d}\lambda_{(i,d)} \beta_i\alpha_1^{d_1} \cdots \alpha_n^{d_n}\]
and $\alpha \in K$. Since $\alpha_i,\beta_i $ are any elements of $K$  for all $i$, $\beta =0$ , $\deg_{\alpha_i}\beta <q$ 
and $\deg_{\beta_i}\beta <q$ it follows that $\lambda_{(i,d)}=0$ for all $i,d$. See Lemma \ref{lemapolcomutativo}.

Therefore, the set $\overline{S}$ is a basis for the quotient vector space $J(X)/T_{\text{Sca}}(UJ_2)$.
Moreover, since $I' \subseteq T_{\text{Sca}}(UJ_2)$, by Lemma \ref{lemageradorGSf} we have 
$I'=T_{\text{Sca}}(UJ_2)$ and the theorem is proved.
\end{proof}

\section{The classical grading }
Let $T_{\text{Cla}}(UJ_2)$ be the set of all $\mathbb{Z}_2$-graded polynomial identities for $UJ_2$ with the classical grading. 
In this section we will describe $T_{\text{Cla}}(UJ_2)$ for any field $K$ of char$(K)=p \neq 2$.

We remember that 
\[UJ_2=\left(UJ_2 \right)_0\oplus \left(UJ_2 \right)_1,\]
where 
\[\left(UJ_2 \right)_0 =span\{e_{11}+e_{22}, \ e_{11}-e_{22} \} \ \mbox{and} \ \left(UJ_2 \right)_1 =span\{e_{12}\}. \]

\begin{lemma}\label{lemaGC}
The polynomials 
\[(y_1,y_2,y_3) , \ \  z_1z_2 \ \ \text{and} \ \   (y_1, z_1, y_2) \]
belong to $T_{\text{Cla}}(UJ_2)$.
\end{lemma}

\begin{proof}
The proof is a direct verification.
\end{proof}

\begin{notation} Let $I$ be the $T_{\mathbb{Z}_2}$-ideal of $J(X)$ generated by the polynomials in Lemma \ref{lemaGC}. 
\end{notation}
Define the equivalence relation $\equiv$ on $J(X)$ as follows: if $f,g \in J(X)$ then
\[f\equiv g \Leftrightarrow f+I=g+I.\]

\begin{lemma}\label{novaidentidadeGC}
The polynomial
\[y_1(y_2(y_3 z_1)) - \frac{1}{2} \left( y_1(z_1(y_2y_3)) + y_2(z_1(y_1y_3))+ y_3(z_1(y_1y_2)) - z_1(y_1(y_2y_3))\right)\]
belongs to $I$.
\end{lemma}

\begin{proof}
Let \[f = 2y_1(y_2(y_3 z_1)) -  y_1(z_1(y_2y_3)) - y_2(z_1(y_1y_3)) - y_3(z_1(y_1y_2)) + z_1(y_1(y_2y_3)).\]
If $u=y_2, \ v=y_3, \ c=z_1$ and $d=y_1$ in (\ref{igualdadebasica2}) we obtain
\begin{align*}
f & = 2((y_3z_1)y_2)y_1 - ((z_1y_2)y_1)y_3 - ((z_1y_3)y_1)y_2 \\
  & = (y_3,z_1,y_2)y_1 + (y_2, z_1y_3, y_1) + (y_3, z_1y_2, y_1).
\end{align*}
Since $(y_1, z_1, y_2) \in I$ it follows that $f \in I$ and so
\[y_1(y_2(y_3 z_1)) - \frac{1}{2} \left( y_1(z_1(y_2y_3)) + y_2(z_1(y_1y_3))+ y_3(z_1(y_1y_2)) - z_1(y_1(y_2y_3))\right) = \frac{1}{2}f \in I.\]
\end{proof}

\begin{lemma}\label{lemavariaveisycomutameassociamGC}
The subalgebra of $J(X)/I$ generated by the set 
\[Y+I=\{y+I \ : \ y\in Y \}\]
is commutative and associative. 
\end{lemma}

\begin{proof}
Let $A_Y$ be the subalgebra of $J(X)/I$ generated by the set $Y+I$. The algebra $J(X)$ is commutative, thus $A_Y$ is commutative too. 
Since  $(y_1,y_2,y_3) \in I$ we have that $A_Y$ is associative, and the lemma is proved.
\end{proof}


\begin{notation}If $Y' = y_{i_1} \cdots y_{i_r}$ and $Y'' = y_{j_1} \cdots y_{j_s}$ are monomials in $J(Y)$ such that
$r,s \geq 0$, $ i_1\leq i_2 \leq \ldots \leq i_r$ and $ j_1\leq j_2 \leq \ldots \leq j_s$, we will write
 $Y' < Y''$ if 
\begin{enumerate}
\item[i)] $r<s$ or
\item[ii)] $r = s$ and  $i_1=j_1, \ i_2=j_2, \ \ldots , \ i_l=j_l, \ i_{l+1}<j_{l+1}$ for some $l$. 
\end{enumerate}

\end{notation}

\begin{lemma}\label{lemageradorGC}
Let $S$ be the subset of $J(X)$ formed by all polynomials
\begin{enumerate}
\item[(a)] $Y'$ and
\item[(b)] $Y'(z_iY'')$,
\end{enumerate}
where $Y'=y_{i_1} \cdots y_{i_r}$; $r\geq 0$; $i_1 \leq \ldots \leq i_r$; $Y''=y_{j_1} \cdots y_{j_s}$; $s\geq 0$; $j_1 \leq \ldots \leq j_s$; $Y' \leq Y''$.
Then the quotient vector space $J(X)/I$ is spanned by the set of all elements $g+I$ where $g\in S$.
\end{lemma}

\begin{proof}
Let $A$ and $B$ be the sets of all elements $g+I$ where $g$ is in (a) and (b), respectively. 
Denote $C=A\cup B$.  If $f$ is a monomial in $J(X)$, we shall prove by induction on deg$(f)$ that
$f+I \in span C$. 

The cases deg$(f)=1$ and deg$(f)=2$ are trivial.

Suppose deg$(f)\geq 3$ and write $f=gh$ where $g,h \in J(X)$ are monomials with degree $<$ deg$(f)$.  By induction hypothesis it follows that $g+I$ and $h+I$ belong to $C$. Since $z_1z_2 \in I$ it is sufficient to consider two cases:

\begin{enumerate}
\item $g+I$ and $h+I$ belong to $A$. 

In this case, $g+I=Y'+I$, $h+I=Y_1'+I$ and  $f\equiv Y'Y_1'$.
By Lemma \ref{lemavariaveisycomutameassociamGC} it follows that $f+I \in A \subset span C$. 

\item $g+I$  belongs to $A$ and $h+I$ belongs to $B$.

In this case, $g+I=Y'+I$ and $h+I=Y_1'(z_iY_2')+I$.  
By Lemma \ref{novaidentidadeGC} 
we obtain
\begin{align}\label{aaa5}
f \equiv &  Y'(Y_1'(z_iY_2')) \nonumber\\
  \equiv &  \frac{1}{2} \left(Y'(z_i(Y_1'Y_2')) + Y_1'(z_i(Y'Y_2')) + Y_2'(z_i(Y'Y_1')) - z_i(Y'(Y_1'Y_2'))\right). 
\end{align}

We will show that $f+I \in span B$.
Firstly, we use Lemma \ref{lemavariaveisycomutameassociamGC} to order the variables in 
$Y_1'Y_2', \ Y'Y_2', \ Y'Y_1', \ Y'(Y_1'Y_2')$ appearing in (\ref{aaa5}). Now, if $Y_3>Y_4$ are monomials in $J(Y)$ then by $(y_1, z_1 , y_2) \in I$
we have $Y_3(z_iY_4) \equiv Y_4(z_iY_3)$. If necessary we can use this to prove that the summands in (\ref{aaa5}) 
are in $B$. 
Thus $f+I \in span B \subset span C$.
\end{enumerate}
The lemma is proved.
\end{proof}

\subsection{The classical grading, when $K$ is an infinite field }

In this subsection we describe the $\mathbb{Z}_2$-graded identities for $UJ_2$ with the classical grading when $K$ is infinite. 

\begin{theorem}\label{teoremaprincipalGC}
If $K$ is an infinite field of char$(K)\neq 2$ then $I=T_{\text{Cla}}(UJ_2)$, that is, $T_{\text{Cla}}(UJ_2)$ is generated, as a $T_{\mathbb{Z}_2}$-ideal, by 
the polynomials in  Lemma \ref{lemaGC}. Moreover, the set in Lemma \ref{lemageradorGC} is a basis for the
quotient vector space $J(X)/I$.
\end{theorem}

\begin{proof}
By Lemma \ref{lemaGC} we have $I \subseteq T_{\text{Cla}}(UJ_2)$.

Let $S$ be the set in Lemma \ref{lemageradorGC}.
Write $\overline{S}=\{g+T_{\text{Cla}}(UJ_2): \ g\in S\}$.
Since $I \subseteq T_{\text{Cla}}(UJ_2)$ we have by Lemma \ref{lemageradorGC} that $J(X)/T_{\text{Cla}}(UJ_2)=span \overline{S}$.

We shall prove that $\overline{S}$ is a linearly independent set. 
Let
\[f(y_1,\ldots , y_r,z_1, \ldots ,z_n)=\sum_{g\in S} \lambda_g g \in T_{\text{Cla}}(UJ_2), \ \lambda_g \in K.\]
Since $K$ is an infinite field, every multihomogeneous component of $f$ belongs to
$T_{\text{Cla}}(UJ_2)$. 

Thus it is sufficient to suppose
\[f(y_1,\ldots ,y_r)= \lambda y_1^{k_1} \cdots y_r^{k_r} \ \mbox{or} \ 
f(y_1,\ldots ,y_r,z_j)= \sum_{l} \lambda_{l} y_1^{k_1 - l_1} \cdots y_r^{k_r - l_r}(z_j(y_1^{l_1} \cdots y_r^{l_r}))\]
where $l = (l_1, \ldots, l_r)$, $0\leq l_i \leq k_i$ for all $i$, and $y_1^{k_1 - l_1} \cdots y_r^{k_r - l_r} \leq y_1^{l_1} \cdots y_r^{l_r}$. We shall prove that $\lambda =
\lambda_{l}=0$ for all $l$.

In the first case, 
$f(1, \ldots, 1) = \lambda \cdot 1=0$
and so $\lambda = 0$.

In the second case, let $Y_i = \alpha_i e_{11} + \beta_i e_{22}$ for all $i$ and let $Z_j = e_{12}$, where $\alpha_i, \beta_i \in K$.
Remember that $(UJ_2)_0=span\{e_{11},e_{22}\}$. 
We have 
\[f(Y_1,\ldots ,Y_r,Z_j)=ue_{12}=0 \ \mbox{where} \ \]
\[u=(1/4)\sum_l \lambda_l\left(\alpha_1^{k_1} \cdots \alpha_r^{k_r} + \alpha_1^{k_1 - l_1} \cdots \alpha_r^{k_r - l_r}\beta_1^{l_1} \cdots \beta_r^{l_r} + \alpha_1^{l_1} \cdots \alpha_r^{l_r}\beta_1^{k_1 - l_1} \cdots \beta_r^{k_r - l_r} + \beta_1^{k_1} \cdots \beta_r^{k_r}\right).\]
Since $y_1^{k_1 - l_1} \cdots y_r^{k_r - l_r} \leq y_1^{l_1} \cdots y_r^{l_r}$, the coefficient of 
$\alpha_1^{k_1 - l_1} \cdots \alpha_r^{k_r - l_r}\beta_1^{l_1} \cdots \beta_r^{l_r}$ in $u$ is $(1/4)\lambda_l$ if 
$l=(l_1,\ldots ,l_r) \neq k=(k_1,\ldots , k_r)$. Since $K$ is infinite and $u=0$ for all $\alpha_i, \beta_i \in K$ it follows that
$\lambda_l=0$ for all $l\neq k$. Now 
\[u=(1/2) \lambda_k\left(\alpha_1^{k_1} \cdots \alpha_r^{k_r}  + \beta_1^{k_1} \cdots \beta_r^{k_r}\right)\]
and with analogous argument we have $\lambda_k=0$ too.

Therefore, the set $\overline{S}$ is a basis for the quotient vector space $J(X)/T_{\text{Cla}}(UJ_2)$.
Moreover, since $I \subseteq T_{\text{Cla}}(UJ_2)$, by Lemma \ref{lemageradorGC} we have 
$I=T_{\text{Cla}}(UJ_2) $.
\end{proof}


\subsection{The classical grading, when $K$ is a finite field }
Throughout this subsection, $K$ is a finite field with $|K| = q$ elements and char$(K)\neq 2$.

\begin{lemma}\label{lemaGCf}
The polynomial $y_1^q - y_1$ belongs to $T_{\text{Cla}}(UJ_2)$.
\end{lemma}

\begin{proof}
In fact, given $Y_1 \in (UJ_2)_0$ we have $Y_1=\alpha e_{11}+\beta e_{22}$ for some $\alpha , \beta \in K$. Since $|K|=q$ we obtain 
$Y_1^q=\alpha^q e_{11}+\beta^q e_{22}=\alpha e_{11}+\beta e_{22}=Y_1$ as desired.
\end{proof}

\begin{notation}
Let $I'$ be the $T_{\mathbb{Z}_2}$-ideal of $J(X)$ generated by the polynomials in Lemmas \ref{lemaGC} and \ref{lemaGCf}.
\end{notation}

\begin{lemma}\label{lemageradorGCf}
Let $\widehat{S}$ be the subset of $J(X)$ formed by all polynomials
\begin{enumerate}
\item[(a)] $Y_1'$ and
\item[(b)] $Y_1'(z_iY_2')$,
\end{enumerate}
where $Y_1'=y_1^{k_1} \cdots y_r^{k_r}$;  $0\leq k_1, \ldots, k_r < q$; $Y_2'=y_1^{l_1} \cdots y_r^{l_r}$; $0\leq l_1, \ldots, l_r < q$; $Y_1' \leq Y_2'$; $r\geq 1$; $z_i \in Z$.
Then the quotient vector space $J(X)/I'$ is spanned by the set of all elements $g+I'$ where $g\in \widehat{S}$.
\end{lemma}

\begin{proof}
Let $A$ and $B$ be the sets of all elements $g+I'$ where $g$ is in (a) and (b), respectively. Denote $C = A \cup B$.
Let $f$ be a monomial in $J(X)$, we shall prove that $f+I' \in span C$. 

Since $I \subseteq I'$ we have by Lemma \ref{lemageradorGC} that the quotient vector space $J(X)/I'$ is spanned by the set of all polynomials:

\begin{enumerate}
\item[(a')] $Y'+I'$,
\item[(b')] $Y'(z_iY'')+I'$,
\end{enumerate}
where $Y'=y_1^{k_1} \cdots y_r^{k_r}$; $0\leq k_1, \ldots, k_r$; $Y''=y_1^{l_1} \cdots y_r^{l_r}$; $0\leq l_1, \ldots, l_r$; $Y' \leq Y''$; $r\geq 0$; $z_i\in Z$.
Let $A'$ and $B'$ be the sets of all elements in (a') and (b'), respectively. We shall prove that $A' \cup B' \subseteq span C$.

\

\noindent \emph{Case 1.} $f+I' \in A'$.

In this case, $f+I' = y_1^{k_1} \cdots y_r^{k_r} + I'$. Since $y_i^q - y_i \in I'$ we can suppose
$0\leq k_1, \ldots, k_r < q$. Thus, $f+I' \in A \subset span C$.

\

\noindent \emph{Case 2.} $f+I' \in B'$.

In this case, $f+I' = (y_1^{k_1} \cdots y_r^{k_r})(z_i(y_1^{l_1} \cdots y_r^{l_r})) + I'$. 
As in Case 1, we can suppose $0\leq k_1, \ldots, k_r < q$ and $0\leq l_1, \ldots, l_r < q$. 
Since $(y_1, z_i , y_2) \in I'$ we obtain 
\[ (y_1^{k_1} \cdots y_r^{k_r})(z_i(y_1^{l_1} \cdots y_r^{l_r})) + I'=(y_1^{l_1} \cdots y_r^{l_r})(z_i(y_1^{k_1} \cdots y_r^{k_r})) + I'.\]
Thus, we can suppose $y_1^{k_1} \cdots y_r^{k_r} \leq y_1^{l_1} \cdots y_r^{l_r}$, and 
consequently $f+I' \in B \subset span C$ as desired.
The lemma is proved.
\end{proof}

\begin{theorem}
If $K$ is a finite field with $|K|=q$ elements and  char$(K)\neq 2$ then $I'=T_{\text{Cla}}(UJ_2)$, that is, $T_{\text{Cla}}(UJ_2)$ is generated, as a $T_{\mathbb{Z}_2}$-ideal, by 
the polynomials in  Lemmas \ref{lemaGC} and \ref{lemaGCf}. Moreover, the set in Lemma \ref{lemageradorGCf} is a basis for the
quotient vector space $J(X)/I'$.
\end{theorem}

\begin{proof}
 By Lemmas \ref{lemaGC} and \ref{lemaGCf} we have $I' \subseteq T_{\text{Cla}}(UJ_2)$.

Consider the set $\widehat{S}$ in Lemma \ref{lemageradorGCf} and write 
$\overline{S}=\{g+T_{\text{Cla}}(UJ_2): \ g\in \widehat{S}\}$.
Since $I' \subseteq T_{\text{Cla}}(UJ_2)$, by Lemma \ref{lemageradorGCf} it follows that $J(X)/T_{\text{Cla}}(UJ_2)=span \overline{S}$.

We shall prove that $\overline{S}$ is a linearly independent set.
Let
\[f(y_1,\ldots ,y_r, z_1,\ldots , z_n)=\sum_{g\in \widehat{S}} \lambda_g g \in T_{\text{Cla}}(UJ_2), \ \lambda_g \in K.\]
In particular,
\[h=f(y_1,\ldots ,y_r, 0, \ldots ,0)= \sum_k \lambda_{k} y_1^{k_1} \cdots y_r^{k_r} \in T_{\text{Cla}}(UJ_2), \]
where $k=(k_1,\ldots ,k_r)$, $0 \leq k_i < q$ for all $i$.
Since $|K| = q$ and $\deg_{y_i} h < q$ for all $i$, we have 
\[h_k(y_1,\ldots y_r)=\lambda_{k} y_1^{k_1} \cdots y_r^{k_r} \in T_{\text{Cla}}(UJ_2)\]
for all $k$. Thus, 
$h_k(1,\ldots ,1)=\lambda_k \cdot 1=0$ and so
$\lambda_k = 0$.

Now, we have 
\[f(y_1,\ldots ,y_r, z_1,\ldots , z_n)=\sum_{i=1}^n\sum_{(l,m)} \lambda_{(l,m)} y_1^{l_1} \cdots y_r^{l_r}(z_i(y_1^{m_1} \cdots y_r^{m_r})) \in T_{\text{Cla}}(UJ_2),\]
where $l=(l_1,\ldots ,l_r)$, $0 \leq l_j < q$ for all $j$, $m=(m_1,\ldots ,m_r)$, $0 \leq m_j < q$ for all $j$, and $y_1^{l_1} \cdots y_r^{l_r} \leq y_1^{m_1} \cdots y_r^{m_r}$. Since $f(y_1,\ldots ,y_r, 0, \ldots ,0, z_i,0, \ldots , 0)\in T_{\text{Cla}}(UJ_2)$ we can suppose
\[f=f(y_1,\ldots ,y_r, z_i)=\sum_{(l,m)} \lambda_{(l,m)} y_1^{l_1} \cdots y_r^{l_r}(z_i(y_1^{m_1} \cdots y_r^{m_r})) \in T_{\text{Cla}}(UJ_2).\]
 Let $Y_j = \alpha_j e_{11} + \beta_j e_{22}$ for all $j$, and $Z_i = e_{12}$, where $\alpha_j, \beta_j \in K$. We have
$f(Y_1, \ldots, Y_r, Z_i)=ue_{12}=0$ where
\[u=(1/4)\sum_{(l,m)}
\lambda_{(l,m)}\left(
 \alpha_1^{l_1 + m_1} \cdots \alpha_r^{l_r + m_r} + \alpha_1^{l_1} \cdots \alpha_r^{l_r}\beta_1^{m_1} \cdots \beta_r^{m_r} + \alpha_1^{m_1} \cdots \alpha_r^{m_r}\beta_1^{l_1} \cdots \beta_r^{l_r} + \beta_1^{l_1 + m_1} \cdots \beta_r^{l_r + m_r}\right).\]
Since $\alpha_i^q=\alpha_i$ and $\beta_i^q=\beta_i$ we can write 
\[u=\sum_{(l,m)}\eta_{(l,m)}\alpha_1^{l_1} \cdots \alpha_r^{l_r}\beta_1^{m_1} \cdots \beta_r^{m_r}=0\]
where $l=(l_1,\ldots ,l_r)$, $0 \leq l_i < q$ for all $i$, $m=(m_1,\ldots ,m_r)$, $0 \leq m_i < q$ for all $i$, $\eta_{(l,m)} \in K$.
In particular, $\eta_{(l,m)}=0$ for all $(l,m)$. Now, if $l\neq (0,\ldots ,0)$ and 
$y_1^{l_1} \cdots y_r^{l_r} \leq y_1^{m_1} \cdots y_r^{m_r}$ we have $(1/4)\lambda_{(l,m)}=\eta_{(l,m)}=0$ and so 
$\lambda_{(l,m)}=0$. In particular,
\[u=(1/2)\sum_{(l,m), l=(0,\ldots ,0)}
\lambda_{(l,m)}\left(
 \alpha_1^{ m_1} \cdots \alpha_r^{m_r}  + \beta_1^{ m_1} \cdots \beta_r^{ m_r}\right)=0.\]
Since 
$0 \leq m_i < q$ for all $i$ we obtain $\lambda_{(l,m)}=0$ if $l=(0,\ldots ,0)$.

Therefore, the set $\overline{S}$ is a basis for the quotient vector space $J(X)/T_{\text{Cla}}(UJ_2)$.
Moreover, since $I' \subseteq T_{\text{Cla}}(UJ_2)$, by Lemma \ref{lemageradorGCf} we have 
$I'=T_{\text{Cla}}(UJ_2) $.
\end{proof}

\section*{Funding}
Mateus Eduardo Salom\~ao was supported by Ph.D. grant from Coordena\c{c}\~ao de Aperfei\c{c}oamento de Pessoal de N\'{i}vel Superior (CAPES). Dimas Jos\'e Gon\c{c}alves was partially supported by Funda\c{c}\~ao de Amparo \`a Pesquisa do Estado de S\~ao Paulo (FAPESP) grant No. 2018/23690-6, and by Conselho Nacional de Desenvolvimento Cient\'{i}fico e Tecnol\'ogico (CNPq) grant No. 406401/2016-0.

\end{document}